\documentclass{article}

\usepackage{amsmath,amsthm,amssymb,enumerate,epsfig,color,graphicx,epstopdf,rotating}

\usepackage[all]{xy}

\usepackage{array}
\newcolumntype{C}[1]{>{\centering\arraybackslash$}p{#1}<{$}}

\usepackage{ifdraft}
\usepackage[colorlinks=true,linkcolor=blue,draft=false]{hyperref}
\usepackage{aliascnt}

\usepackage{algorithm}
\usepackage[noend]{algorithmic}

\def\<{\langle}
\def\>{\rangle}

\theoremstyle{plain}
\newtheorem{theorem}{Theorem}[section]

\newaliascnt{lemma}{theorem}
\newtheorem{lemma}[lemma]{Lemma}
\aliascntresetthe{lemma}

\newaliascnt{proposition}{theorem}
\newtheorem{proposition}[proposition]{Proposition}
\aliascntresetthe{proposition}

\newaliascnt{corollary}{theorem}
\newtheorem{corollary}[corollary]{Corollary}
\aliascntresetthe{corollary}

\newaliascnt{conjecture}{theorem}

\aliascntresetthe{conjecture}

\theoremstyle{remark}

\newaliascnt{claim}{theorem}

\aliascntresetthe{claim}

\newtheorem*{claim*}{Claim}
\newtheorem*{remark}{Remark}

\theoremstyle{definition}

\newaliascnt{definition}{theorem}
\newtheorem{definition}[definition]{Definition}
\aliascntresetthe{definition}

\newaliascnt{example}{theorem}

\aliascntresetthe{example}

\newaliascnt{notation}{theorem}

\aliascntresetthe{notation}

\ifdraft{
\newcommand{\colorF}[1]{{\color{blue}#1}}
\newcommand{\colorJ}[1]{{\color{red} #1}}
\newcommand{\colorL}[1]{{\color[rgb]{0,0.7,0} #1}}
\newcommand{\commF}[1]{\marginpar{\tiny\vskip-3ex\colorF{V: #1}}}
\newcommand{\commJ}[1]{\marginpar{\tiny\vskip-3ex\colorJ{J: #1}}}
\newcommand{\commL}[1]{\marginpar{\tiny\vskip-3ex\colorL{M: #1}}}
\newcommand{\comment}[1]{\marginpar{\tiny\vskip-3ex\color[rgb]{0.5,0.5,0.5} #1}}
}
{
\newcommand{\colorF}[1]{#1}
\newcommand{\colorJ}[1]{#1}
\newcommand{\colorL}[1]{#1}
\newcommand{\commF}[1]{}
\newcommand{\commJ}[1]{}
\newcommand{\commL}[1]{}
\newcommand{\comment}[1]{}
}

\title{On lexicographic representatives in braid monoids}
\author{\colorF{Ram\'on Flores}, \colorJ{Juan Gonz\'alez-Meneses}\footnote{Both authors partially supported by Spanish Project MTM2016-76453-C2-1-P and FEDER.}}
\date{August, 2018}

\begin{document}

\maketitle


\begin{abstract}
The language of maximal lexicographic representatives of elements in the positive braid monoid $A_n$ with $n$ generators is a regular language. We describe with great detail the smallest Finite State Automaton accepting such language, and study the proportion of elements of length $k$ whose maximal lexicographic representative finishes with the first generator. This proportion tends to some number $P_{n,1}$, as $k$ tends to infinity, and we show that $P_{n,1}\geq \frac{1}{8}$ for every $n\geq 1$. We also provide an explicit formula, based on the Fibonacci numbers, for the number of states of the automaton.
\end{abstract}

\section{Introduction}\label{S:Introduction}

The positive braid monoid with $n$ generators, $A_n$, also known as the spherical type Artin-Tits monoid of type $A_n$, or as the positive braid monoid on $n+1$ strands $\mathcal B_{n+1}$, is the monoid with the following presentation:
$$
   A_n=\left\langle a_1,\ldots,a_n \left| \begin{array}{cl} a_i a_j = a_j a_i & \mbox{ if } |j-i|>1 \\ a_i a_j a_i = a_j a_i a_j & \mbox{ if } |j-i|=1 \end{array} \right. \right\rangle
$$

Since the appearance of the seminal survey included in \cite{CEHLPT}, which in turn reunited different notions that have circulated some years around the group-theoretic community, there has been a growing interest in the role of Finite State Automata in Group Theory. In particular, in the context of Artin groups and monoids, different automata have been proposed, see for example \cite{Cha92} or \cite{Jug16}.

In the present paper, we deeply analyze a Finite State Automaton proposed by the second author and V. Gebhardt in~\cite{GG}. In that paper, a polynomial algorithm is given to select a random element of $A_n$, among all elements of given length, with uniform probability. The difficulty of this task comes from the fact that a given positive braid $\beta$ may admit many different words in the generators $a_1,\ldots,a_n$ representing it (although all representatives have the same length, as the relations in $A_n$ are homogeneous). In order to be able to count the elements in $A_n$ of length $k$, one can choose a unique representative for each element. For instance, one can order lexicographically the (finite set of) words representing $\beta$, setting $a_1<a_2<\cdots <a_n$, and choose the smallest element. With this ordering, it is shown in~\cite{GG} that the language $L_{n}$ of smallest lexicographic representatives of elements of $A_n$ is a regular language (a fact that was previously known), and the mentioned Finite State Automaton $\Gamma_{n}$ accepting this language is defined. It is proved in the mentioned paper that $\Gamma_{n}$ is the smallest (in terms of number of states) Deterministic Finite State Automaton accepting $L_{n}$. Actually, in~\cite{GG}, the language and the automaton for $A_n$ are denoted $L_{n+1}$ and $\Gamma_{n+1}$, respectively, as the index indicates the number of strands instead of the number of generators, but in the present paper we prefer to shift the indices to simplify the notations.

Our immediate motivation to study the mentioned Finite State Automaton comes from our paper~\cite{FG}. One of our interests there is to determine the limit of the growth rates of the monoids $A_n$, as $n$ tends to infinity, a computation that in particular builds new and exciting bridges between Group Theory and Combinatorics in one side, and Real Analysis and Modular Forms theory on the other, via the partial theta functions. In our computation, it turns out to be necessary to describe the proportion of lexicographic representatives in the monoid $A_n$ that finish with a given generator. According to the convention in~\cite{FG}, we will consider the {\it biggest} lexicographic representatives instead of the smallest (as in~\cite{GG}), and call $\mathcal L_n$ the corresponding language. It is easy to see that the smallest Finite State Automaton $\Gamma_n$ accepting $L_n$, studied in~\cite{GG}, coincides with the smallest Finite State Automaton accepting $\mathcal L_n$, just replacing each $a_i$ with $a_{n+1-i}$ in all transitions. More information in \autoref{S:local description} below.

From now on, $\Gamma_n$ will denote the smallest Deterministic Finite State Automaton accepting the language of {\it biggest} lexicographic representatives in $A_n$, with $a_1<a_2<\cdots< a_n$. It will be the detailed description of this automaton $\Gamma_n$ which will allow us to compute the desired proportion, following a reasoning that we briefly sketch now.

The proportion of elements in the language $\mathcal L_n$  finishing with a given generator has a sense only if we fix the length of the words, so let $\mathcal L_{n,k}$ be the (finite) set of words in $\mathcal L_n$ having length $k$. We see that $\mathcal L_{n,0}=\{\epsilon\}$ (the empty word) and $\mathcal L_{n,1}=\{a_1,\ldots,a_n\}$. Also,
$$
   \mathcal L_{n,2}=\{a_1a_1, a_1a_2\}\cup \{a_2a_1, a_2a_2, a_2a_3\} \cup \cdots \cup \{a_na_1, a_na_2,\cdots, a_na_n\}.
$$
Notice that the words $a_ia_j$ with $j\geq i+2$ do not appear, as in that case $a_ia_j=a_ja_i$ in $A_n$, and the latter is the biggest lexicographic representative. The case of $\mathcal L_{n,3}$ is more involved, as we must also take into account the relations $a_ia_{i+1}a_i=a_{i+1}a_ia_{i+1}$.

Recall that $\mathcal L_n$ is a regular language. This implies that, when $k$ tends to infinity, the proportion of words in $\mathcal L_{n,k}$ which end at a given state $s$ of the automaton tends to a well defined limit $0\leq p_s <1$. The set of numbers $p_s$ for all states $s$ in $\Gamma_n$ is called the stationary distribution of $\Gamma_n$. We can understand $p_s$ as the probability that a very long word in $\mathcal L_n$ finishes at the state $s$ (considering all words of given length in $\mathcal L_n$ with uniform probability).

We will see that all words in $\mathcal L_n$ which end at a given state in $\Gamma_n$ finish with the same letter. That is, if we see $\Gamma_n$ as a directed graph, whose arrows are labeled by the generators $a_1,\ldots,a_n$, all the incoming arrows of a given state $s$ have the same label. It follows that, when $k$ tends to infinity, the proportion of words of length $k$ finishing with $a_i$, tends to the sum of the numbers $p_s$ for all states $s$ whose incoming arrows are labeled by $a_i$. This limit is thus well defined, and it is the probability that a very long word in $\mathcal L_n$ finishes with $a_i$. Let us denote it by $P_{n,i}$.

It seems natural that for each $n>2$, the proportion $P_{n,1}$ is greater than $P_{n,i}$ for $i\geq 1$. The question arises about how far is this proportion from the uniform distribution $\frac{1}{n}$. It turns out that the lexicographic representatives ending with $a_1$ are much more abundant than one could think, since the proportion $P_{n,1}$ has a uniform lower bound, independent of $n$. We can show the following:

\autoref{T:generic_braids_end_with_a1}.
{\it Let $P_{n,1}$ be the limit, when $k$ tends to infinity, of the proportion of maximal lexicographic representatives of length $k$ in $A_n$ finishing with $a_1$. Then $P_{n,1}>\frac{1}{8}$ for every $n\geq 1$.
}

Our exhaustive description of the automaton $\Gamma_n$ also permits to compute in an exact way the number of states, by means of a formula which surprisingly depends on the Fibonacci numbers.

\autoref{T:numer_of_states} {\it Let $F_n$ be the $n$-th Fibonacci number $(F_0=0$, $F_1=1$, $F_{i+1}=F_i+F_{i-1})$. Let $s_n$ be the number of states of the Finite State Automaton $\Gamma_n$ accepting the language of maximal (or minimal) lexicographic representatives of elements in the positive braid monoid $A_n$. Then
$$
   s_n=\sum_{i=1}^n{\left({n+1-i \choose 2}+1\right)\cdot F_{2i}}
$$
}

Note that it is shown in \cite[Corollary 4.17]{GG} that $\Gamma_n$ has at least $2^{n-1}$ states, but no upper bound for the number of states is given there.

We finish the Introduction by briefly explaining the contents of the paper. In \autoref{S:local description} we present the automaton and describe some easy properties of it, while in \autoref{S:description}, the core of the paper, a complete analysis of it is undertaken. \autoref{S:size} is devoted to the computation of the number of states of the automaton; and \autoref{S:asymptotic}, which concludes the paper, makes use of the previous results in the paper and some linear algebra from Perron-Frobenius theory, to obtain asymptotic properties of the braid monoids, and in particular the desired results concerning proportions. We end the paper with an appendix showing how to compute the incidence matrix of the automaton $\Gamma_n$, and providing the numerical results of our computations.

\section{The automaton}\label{S:local description}

A standard reference about finite state automata in group theory is \cite{CEHLPT}, and a modern approach can be found in \cite{HRR17}.

In this paper, a finite state automaton $\Gamma$ will be given as a directed graph, whose vertices will be called {\it states}, and whose arrows, called {\it transitions}, are labeled by letters from a finite alphabet $\{a_1,\ldots,a_n\}$. There will be one special vertex called the {\it initial state}. $\Gamma$ is a Deterministic Finite State Automaton (DFSA) if for every state $s$ and every $i\in \{1,\ldots,n\}$, there is at most one arrow starting at $s$ and labeled $a_i$. All states in $\Gamma$ will be considered to be {\it accepted}.

A {\it path} $\alpha$ in $\Gamma$ is a sequence of arrows $\alpha_1,\ldots,\alpha_t$ such that the target of $\alpha_i$ equals the source of $\alpha_{i+1}$, for $i=1,\ldots,t-1$. The source of the path $\alpha$ is the source of $\alpha_1$, and the target of $\alpha$ is the target of $\alpha_t$. The word associated to $\alpha$ is the concatenation of the labels of $\alpha_1,\alpha_2,\ldots,\alpha_t$. A word $w$ in the alphabet $\{a_1,\ldots,a_n\}$ is {\it accepted} by $\Gamma$ if there is a path in $\Gamma$ whose source is the initial state, and whose associated word is $w$. The set of accepted words form the {\it language} accepted by $\Gamma$.

Given the way in which we are going to deal with the natural inclusions $A_1\subset A_2 \subset A_3 \subset \cdots$ and the notation of \cite{FG}, it will be convenient for us to consider the {\it biggest} lexicographic representatives, instead of the smallest as in \cite{GG}. This is not a big issue, as the map $\rho$ sending $a_i$ to $a_{n+1-i}$ is a monoid automorphism of $A_n$ (as it can be immediately deduced from the monoid presentation) and, applied to words, it sends the smallest lexicographic representative of a braid $\beta$ to the biggest lexicographic representative of $\rho(\beta)$, and vice-versa. In other words, the language of biggest representatives $\mathcal L_n$ for elements in $A_n$ with $a_1<a_2<\cdots <a_n$ corresponds (via $\rho$) to the language of smallest representatives, also with $a_1<a_2<\cdots <a_n$. So, for every $n\geq 1$, we will use the definition of the DFSA in~\cite{GG} to describe a DFSA $\Gamma_n$, with alphabet $\{a_1,\ldots,a_n\}$, whose accepted language is the set $\mathcal L_n$ of biggest lexicographic representatives of elements of $A_n$.

Recall that $A_n$ is the monoid of positive braids with $n$ generators. In this monoid we say that $\beta_1$ is a prefix of $\beta$, and we write $\beta_1\preccurlyeq \beta$, if there exists $\beta_2\in A_n$ such that $\beta=\beta_1\beta_2$.
Given a braid $\beta\in A_n$, we denote by $\omega(\beta)\in \mathcal L_n$ its biggest lexicographic representative. Conversely, given a word $w$ in the alphabet $\{a_1,\ldots,a_n\}$, we will denote by $\overline w\in A_n$ the braid it represents. It is clear that $\overline{\omega(\beta)}=\beta$ for every $\beta\in A_n$, and that $\omega(\overline{w})=w$ if and only if $w\in \mathcal L_n$.

Given a word $w$ in the alphabet $\{a_1,\ldots,a_n\}$, we define the {\it support} of $w$, denoted $\mbox{supp}(w)$, as the set of letters which appear in $w$. Notice that applying a relation from the presentation of $A_n$ does not change the support of a word. It follows that all words representing a given braid $\beta$ have the same support, so we can talk about the {\it support} of $\beta$, denoted $\mbox{supp}(\beta)$, as the support of any of its representatives. Also, applying a relation from the presentation of $A_n$ does not change the length of a word, so all representatives of a braid have the same length, and we can talk about the {\it length} $|\beta|$ of a braid $\beta$, as the length of any of its representatives.

Given $\beta\in A_n$, we say that a decomposition $\beta=\beta_1\beta_2$ is {\it permitted} if $\omega(\beta) =\omega(\beta_1) \omega(\beta_2)$, and is {\it forbidden} otherwise. Notice that, if $|\beta|=k$, there is exactly one permitted decomposition of $\beta$ for each number $k_1=0,\ldots,k$. Indeed, the word $w=\omega(\beta)$ has a unique prefix $w_1$ of length $k_1$. Hence, if we write $w=w_1 w_2$, we have that $\beta=\overline{w_1}\; \overline{w_2}$ is the only permitted decomposition whose left factor has length $k_1$.

For example, if we consider the braid $\beta = a_1a_2a_1\in A_n$, the decomposition given by $\beta_1=a_1$ and $\beta_2=a_2a_1$ is forbidden, because $\omega(\beta_1)=a_1$, $\omega(\beta_2)=a_2a_1$, but $\omega(\beta)=a_2a_1a_2$, which is not the same word as $\omega(\beta_1) \omega(\beta_2) = a_1a_2a_1$. The only permitted decompositions of $\beta$ are those obtained from prefixes of its lexicographic representative $a_2a_1a_2$.

Notice that the biggest lexicographic representative of a braid $\beta$ is determined by its permitted decompositions. Hence, it is also determined by its forbidden decompositions.

If a decomposition $\beta=\beta_1\beta_2$ is forbidden, we say that $\beta_2$ is {\it forbidden after} $\beta_1$. It is important to notice that if $\beta_2$ is forbidden after $\beta_1$, then $\beta_2\beta_3$ is also forbidden after $\beta_1$, for every $\beta_3\in A_n$. Therefore, in order to determine all braids which are forbidden after $\beta_1$, it suffices to know those which are {\it minimal} with respect to the prefix order. This justifies the following definition.

\begin{definition}\cite{GG}
Given a braid $\beta_1\in A_n$, the set of {\it minimal forbidden prefixes} after $\beta_1$ is:
$$
   F_n(\beta_1)=\min_{\preccurlyeq}\{\beta_2\in A_n;\ \omega(\beta_1\beta_2)\neq \omega(\beta_1)\omega(\beta_2)\}.
$$
\end{definition}

As we mentioned before, $F_n(\beta_1)$ determines the whole set of forbidden braids after $\beta_1$, which in turn determines the set of words in $\mathcal L_n$ that start with $\omega(\beta_1)$. The advantage of this notion is that there is only a finite number of sets of the form $F_n(\beta_1)$, as we will see. Therefore, we can construct a DFSA which accepts the language $\mathcal L_n$ in the following way:

\begin{definition}\cite{GG}
The automaton $\Gamma_n$ is a directed graph defined as follows.
\begin{itemize}
 \item The vertices of $\Gamma_n$ correspond to the sets of the form $F_n(\beta_1)$ for $\beta_1\in A_n$.

 \item Whenever $a_j\notin F_n(\beta_1)$, there is an arrow labeled $a_j$ whose source is $F_n(\beta_1)$ and whose target is $F_n(\beta_1 a_j)$.
\end{itemize}
\end{definition}

\begin{theorem}{\rm \cite{GG}}\label{T:automaton}
$\Gamma_n$ is a well defined DFSA, whose accepted language is $\mathcal L_n$. Moreover, it is the smallest DFSA (in terms of number of states) accepting $\mathcal L_n$.
\end{theorem}

We have an immediate consequence of the above result:

\begin{corollary}\label{C:permitted_decomposition}
Let $\beta_1,\beta_1',\beta_2\in A_n$. If $\beta_1\beta_2$ is a permitted decomposition, and $F_n(\beta_1)=F_n(\beta_1')$, then $\beta_1'\beta_2$ is a permitted decomposition and $F_n(\beta_1'\beta_2)=F_n(\beta_1\beta_2)$.
\end{corollary}

\begin{proof}
Let $\beta=\beta_1\beta_2$, $w=\omega(\beta)$, $w_1=\omega(\beta_1)$ and $w_2=\omega(\beta_2)$. Since the decomposition $\beta=\beta_1\beta_2$ is permitted, we have $w=w_1w_2$. We can then read the word $w_1w_2$ in the automaton $\Gamma_n$, starting at the initial state. After reading $w_1$, we are placed at the state $F_n(\beta_1)$. The final state ($F_n(\beta)$) only depends on the current state $F_n(\beta_1)$, and the remaining part of the word, $w_2$.

Now let $w_1'=\omega(\beta_1')$. Since $F_n(\beta_1)=F_n(\beta_1')$, it follows that there is a path in $\Gamma_n$ labeled by $w_1'$, starting at the initial state and finishing at $F_n(\beta_1)$. If we continue reading the word $w_2$, we will finish at the state $F_n(\beta)$. Hence, there is a path in $\Gamma_n$ labeled by $w_1'w_2$, starting at the initial state and finishing at $F_n(\beta)$. This implies that $w_1'w_2\in \mathcal L_n$, that $\beta_1'\beta_2$ is permitted, and that $F_n(\beta_1'\beta_2)=F_n(\beta)=F_n(\beta_1\beta_2)$.
\end{proof}

\section{Description of the automaton}\label{S:description}

We want to describe the automaton $\Gamma_n$ in detail. Hence, we will describe the states $F_n(\beta_1)$, and the transitions $F_n(\beta_1)\stackrel{a_j}{\longrightarrow} F_n(\beta_1a_j)$.

First of all, it is clear that $F_n(1)=\emptyset$. Then, one can compute all possible sets of minimal forbidden prefixes, with the help of the following crucial result:

\begin{proposition}{\rm \cite{GG}}\label{P:Transitions}
  Let $\beta_1\in A_n$, and suppose that $a_j\notin F_n(\beta_1)$, so there is an arrow $F_n(\beta_1)\stackrel{a_j}{\longrightarrow} F_n(\beta_1a_j)$ in the automaton $\Gamma_n$. One has:
\begin{enumerate}

\item If $a_i\in F_n(\beta_1)$ for some $i<j-1$, then $a_i\in F_n(\beta_1a_j)$.

\item If $a_i a_{i+1}\cdots a_{j-1}\in F_n(\beta_1)$ for some $i<j$, then  $a_i a_{i+1}\cdots a_{j}\in F_n(\beta_1 a_j)$.

\item If $a_ia_{i+1}\cdots a_k\in F_n(\beta_1)$, with $i<j\leq k$, then $a_ia_{i+1}\cdots a_k\in F_n(\beta_1 a_j)$.

\item If $a_ja_{j-1}\in F_n(\beta_1)$, then $a_{j-1}\in F_n(\beta_1a_j)$.

\item If $a_ja_{j+1}\cdots a_k\in F_n(\beta_1)$, with $j<k$, then $a_{j+1}\cdots a_k\in F_n(\beta_1a_j)$.

\item If $a_ja_{j+1}\notin F_n(\beta_1)$, then $a_{j+1}a_j \in F_n(\beta_1 a_j)$ (provided $j\leq n-1$).

\item $a_{j+2},a_{j+3},\ldots, a_n \in F_n(\beta_1a_j)$ (provided $j\leq n-2$).

\end{enumerate}

Moreover, all elements of $F_n(\beta_1a_j)$ are obtained in this way.

\end{proposition}

Let us think about the conditions in \autoref{P:Transitions}. First of all, they imply that all elements in $F(\beta_1a_j)$ are of the form $a_ia_{i+1}\cdots a_k$ for some $i\leq k$, with the only possible exception of the braid $a_{j+1}a_j$. Therefore, as these braids can only be written in a unique way, it makes sense to talk about {\it the starting letter} of such a braid. Conditions~1 to 4 tell us which elements in $F(\beta_1 a_j)$ start with a letter smaller than $a_j$. Such an element is either a single letter or an increasing sequence of consecutive letters whose final letter is at least $a_j$. Conditions~5 and 6 produce the elements in $F_n(\beta_1a_j)$ starting with $a_{j+1}$. Actually, Condition~6 tells us that $a_{j+1}a_j$ is a minimal forbidden prefix, except when $j=n$ or when Condition~5 produces a single letter: in this case $a_{j+1}$ is forbidden and hence $a_{j+1}a_j$ is not minimal. Finally, Condition~7 tells us that the letters $a_{j+2},a_{j+3},\ldots,a_{n}$ always belong to $F_n(\beta_1 a_j)$, provided they exist.

%
%
%
%
%
%

Let us deduce an important property from \autoref{P:Transitions}, i.e. that all arrows in $\Gamma_n$ whose target is a given state have the same label. We can show that the final letter of a word $\omega(\beta)$ depends only on the set $F_n(\beta)$, and not on $\beta$ itself:

\begin{proposition}
Let $F=F_n(\beta)$ be the set of minimal forbidden prefixes after some nontrivial $\beta\in A_n$. Let $t$ be the biggest index such that there is no element in $F$ starting with $a_t$. Then $a_t$ is the final letter of $\omega(\beta)$.
\end{proposition}

\begin{proof}
Let $w=\omega(\beta)$. Since $\beta$ is not trivial, we can write $w=w_1a_j$ for some $w_1\in \mathcal L_n$. We need to show that $j=t$.

Let $\beta_1=\overline{w_1}$. Since $w_1,w\in \mathcal L_n$, it follows that $a_j$ is not forbidden after $\beta_1$, and there is an arrow labeled $a_j$ from $F_n(\beta_1)$ to $F_n(\beta)$. From \autoref{P:Transitions}, $F_n(\beta)$ contains all letters $a_{j+2},\ldots, a_{n}$ (provided $j\leq n-2$). Also, it contains either $a_{j+1}$ or $a_{j+1}a_j$ (provided $j\leq n-1$). And it contains no element starting with $a_j$. Therefore, $j$ is the biggest index such that there is no element in $F_n(\beta)$ starting with $a_j$.
\end{proof}

We can then talk about the {\it final letter} of a state $s$ in $\Gamma_n$, meaning the final letter of any word in $\mathcal L_n$ ending at $s$, or the label of any arrow in $\Gamma_n$ whose target is $s$. We can also talk about the {\it final letter} of a set $F_n(\beta)$ of minimal forbidden prefixes, as the final letter of $\omega(\beta)$.

We will now see that we can describe every set of the form $F_n(\beta)$ by a collection of data, that we will call a {\it segment configuration}.

\begin{definition}
A {\it segment configuration} for $A_n$ is a $4$-tuple $(i,j,k,S)$ where $i,j,k$ are integers, $1\leq i\leq j\leq k\leq n$, and
$$
S=\{[i_1,j_1],[i_2,j_2],\ldots,[i_t,j_t]\}
$$
is a (possibly empty) collection of ``nested" segments with integer endpoints, where
$$
 i \leq i_1< i_2 < \cdots <i_t < j \leq k\leq j_{t-1} \leq \cdots \leq j_2\leq j_1\leq n,
$$
and either $j_t=j$ or $k\leq j_t\leq j_{t-1}$.
\end{definition}

Each segment configuration determines a set of braids:

\begin{definition}\label{segmentconfig}
Let $\mathcal S_n$ be the set of segments configurations for $A_n$. We define a map
$$
   \psi_n:\: \mathcal S_n \longrightarrow \mathcal P(A_n)
$$
as follows: given $(i,j,k,S)\in \mathcal S_n$ with $S=\{[i_1,j_1],[i_2,j_2],\ldots,[i_t,j_t]\}$, we set
$$
  \begin{array}{rcl} \psi_n((i,j,k,S)) & = & \{a_{i_r}a_{i_r+1}\cdots a_{j_r}; \quad  r=1,\ldots,t\} \\
    & \cup & \{a_r;\quad i\leq r \leq n,\  r\neq i_1,\ldots,i_t,j,j+1\} \\
    & \cup & U_{j,k},
    \end{array}
$$
where
$$
    U_{j,k}=\left\{\begin{array}{ll} \emptyset & \mbox{ if }k=j=n, \\
    \{a_{j+1}a_j\} & \mbox{ if }k=j<n, \\
    \{a_{j+1}\} & \mbox{ if }k=j+1, \\
    \{a_{j+1}a_j, \; a_{j+1}a_{j+2}\cdots a_k\} & \mbox{ if }k>j+1.
    \end{array}
    \right.
$$
\end{definition}

\medskip

As said in the introduction, one of the main goals of this paper is a detailed study of the automaton $\Gamma_n$. This goal will be achieved by identifying segment configurations with sets of minimal forbidden prefixes after positive braids. More precisely, we will describe the map $\psi_n$ just defined, following the next steps:

\begin{itemize}

\item $\psi_n$ is injective, and hence different segment configurations give rise to different sets of braids (\autoref{phiinjective}).

\item Every set of the form $F_n(\beta)$ is in the image of $\psi_n$ (\autoref{P:necessary_conditions}).

\item Every set in the image of $\psi_n$ is equal to $F_n(\beta)$ for some braid $\beta$ (\autoref{P:bijection}).

\end{itemize}

This precise knowledge of $\psi_n$ is essentially equivalent to a deep understanding of the structure of the automaton, as we will see in the proof of \autoref{P:bijection}. From now, we will undertake the previous program.

\begin{lemma}\label{phiinjective}
The map $\psi_n$ is injective.
\end{lemma}

\begin{proof}
Suppose that $\psi_n((i,j,k,S))=\psi_n((i',j',k',S'))=F$.

Assume first that $F=\emptyset$. Then $U_{j,k}=\emptyset$, so $j=k=n$. Also, if $i<j$ there would be an element in $F$ starting with $a_i$, hence $i=j=k=n$. Now every segment $[i_r,j_r]\in S$ must satisfy $i\leq i_r<j$, which is not possible if $i=j$. Hence $(i,j,k,S)=(n,n,n,\emptyset)$. In the same way $(i',j',k',S')=(n,n,n,\emptyset)$, so the empty set has a unique preimage under $\psi_n$.

Assume now that $F\neq \emptyset$. By construction, all elements in $F$ have the form $a_pa_{p+1}\cdots a_q$ for some $p\leq q$, or $a_{p+1}a_{p}$. Hence it makes sense to talk about the {\it starting letter} of an element in $F$.

We see that $j=j'$ as it is the biggest index $p$ for which there is no element in $F$ starting with $a_p$.

Let $p$ be the smallest index for which there is an element in $F$ starting with $a_p$. Notice that $p\leq j+1$ and $p\neq j$. If $p<j$, it follows from the definition of $\psi_n$ that $i<j$ and that $i=p$. On the other hand, if $p=j+1$, it follows that $i=j$. Applying the same arguments to $i'$, we see that if $p<j$ we have $i'=p$, and if $p=j+1$ we have $i'=j$. In any case, $i=i'$.

It is also true that $k=k'$, as this value is determined by the elements in $F$ starting with $a_{j+1}$.

We finish by observing that $S=S'$, as their segments correspond to the elements in $F$, longer than one letter, that start with letters smaller than $a_j$.
\end{proof}

\begin{proposition}{\rm \cite{GG}}\label{P:necessary_conditions}
Let $F=F_n(\beta)$ be the set of minimal forbidden prefixes after some nontrivial $\beta\in A_n$. Then $F=\psi_n((i,j,k,S))$ for some segment configuration $(i,j,k,S)\in \mathcal S_n$. Moreover, $j$ is the final letter of $F$.
\end{proposition}

\begin{proof}
This result is a direct consequence of \cite[Corollary 4.12]{GG}. It can also be shown directly from \autoref{P:Transitions} starting with $F_n(1)=\emptyset$, by showing that if a set $F_n(\beta_1)$ can be defined by a segment configuration, and $a_m$ is not forbidden after $\beta_1$, then $F_n(\beta_1a_m)$ can also be defined by a segment configuration.
\end{proof}

The above result states that every set of minimal forbidden prefixes can be determined by a segment configuration. That is, $\mathcal F_n\subset \mbox{Im}(\psi_n)$, where $\mathcal F_n=\{F_n(\beta);\ \beta\in A_n\}$.  As $\psi_n$ is injective, this defines an injective map: $\varphi_n: \mathcal F_n \rightarrow \mathcal S_n$, which, as we will see, will yield the desired bijection between the sets of forbidden prefixes and their corresponding segment configurations.

\begin{figure}
\begin{center}
  \includegraphics{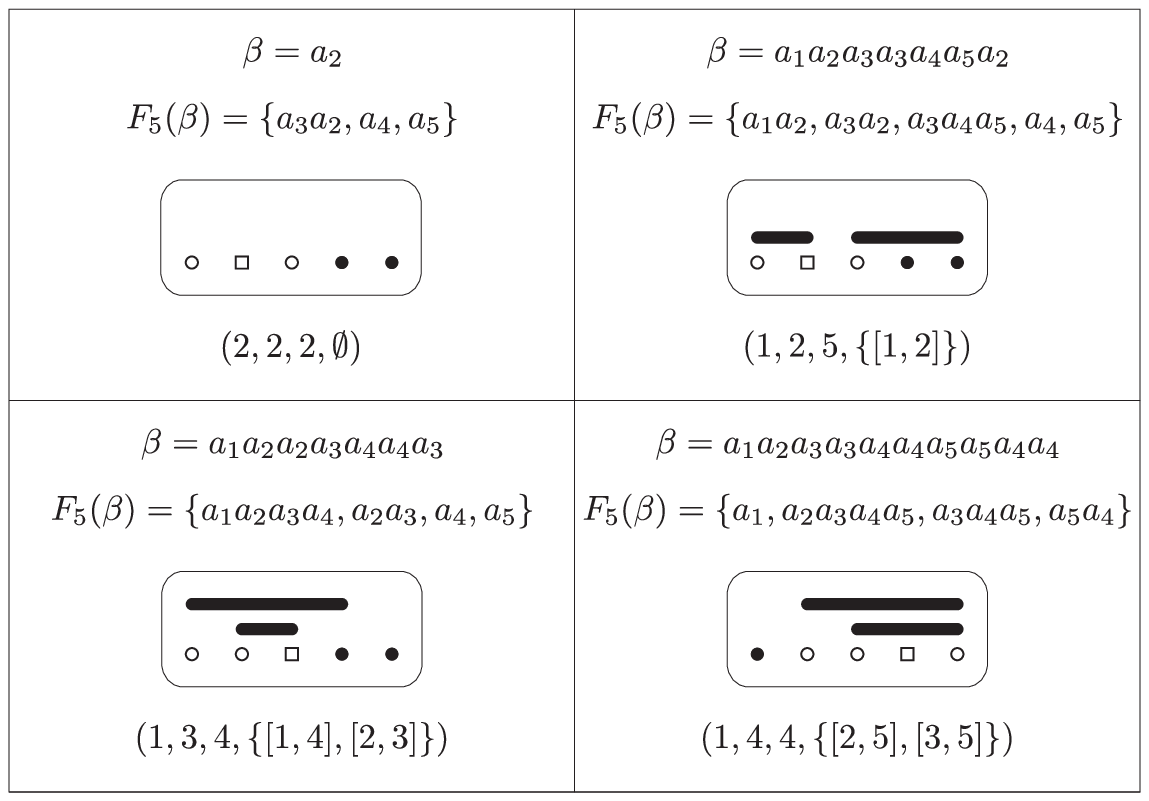}
\end{center}
\caption{Examples of braids in $A_5$, their sets of minimal forbidden prefixes, their diagrams, and their segment configurations.}
\label{F:diagrams_examples}
\end{figure}

We will represent a segment configuration using a diagram, as in \autoref{F:diagrams_examples}. Given a segment configuration $(i,j,k,S)$, and its corresponding set of braids $F=\psi_n((i,j,k,S))$, the diagram will contain the following: a little white square at position $j$; a black circle at position $r$ for each $a_r\in F$; a segment from $p$ to $q$ ($p<q$) for every $a_pa_{p+1}\cdots a_q$ in $F$; white circles at all positions where there are no black circles and no white square.

In \autoref{F:diagrams_examples} we can see several examples of diagrams for segment configurations corresponding to sets of forbidden prefixes $F_n(\beta)$. Notice that a letter $a_r$ is permitted after $\beta$ if and only if there is no black circle at position $r$.

We will make extensive use of \autoref{P:Transitions}, in order to describe the graph $\Gamma_n$. All the arguments will become simpler if we work with diagrams instead of with sets of braids. Therefore, we will translate \autoref{P:Transitions} to the language of diagrams. More precisely, we will describe transitions between diagrams in such a way that, if the source diagram corresponds to a set of forbidden prefixes, the transition corresponds to the one described in \autoref{P:Transitions}. In the following pictures, a small dot means that we do not care about the actual symbol in that position (white or black circle, or white square).

\begin{definition}\label{D:Transitions}
Let $(a,b,c,S)$ be a segment configuration, represented by a diagram $D$. Suppose that there is no black circle at position $j$ in $D$. Then we define a transition
$$
(a,b,c,S) \stackrel{a_j}{\longrightarrow} (p,j,q,S'),
$$
where $(p,j,q,S')$, represented by a diagram $D'$, is defined the following way:
\begin{enumerate}

 \item $D'$ has a square at position $j$.

 \item If $D$ has a black circle at $i<j-1$, then $D'$ has a black circle at $i$.

  \begin{center}
  \includegraphics{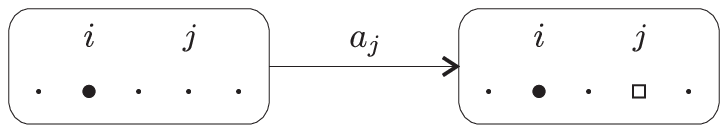}
  \end{center}

 \item If $D$ has a segment $[i,j-1]$, then $D'$ has a segment $[i,j]$.

  \begin{center}
  \includegraphics{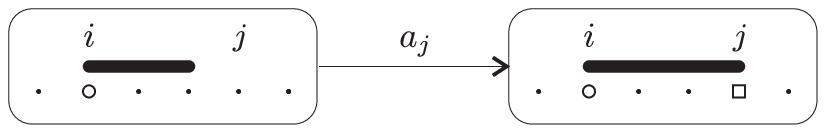}
  \end{center}

 \item If $D$ has a segment $[i,k]$ with $i<j\leq k$, then $D'$ has a segment $[i,k]$.

  \begin{center}
  \includegraphics{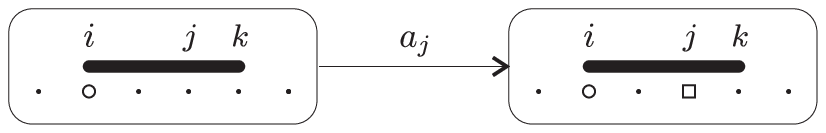}
  \end{center}

 \item If $D$ has a square at $j-1$, then $D'$ has a black circle at $j-1$.

  \begin{center}
  \includegraphics{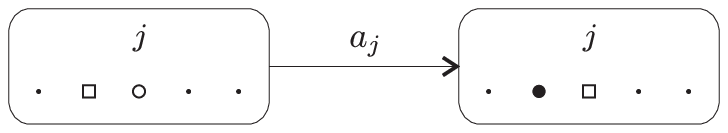}
  \end{center}

 \item If $D$ has a segment $[j,k]$, then $D'$ has either a segment $[j+1,k]$ (if $j+1<k$) or a black circle at     $k$ (if $j+1=k$).

   \begin{center}
  \includegraphics{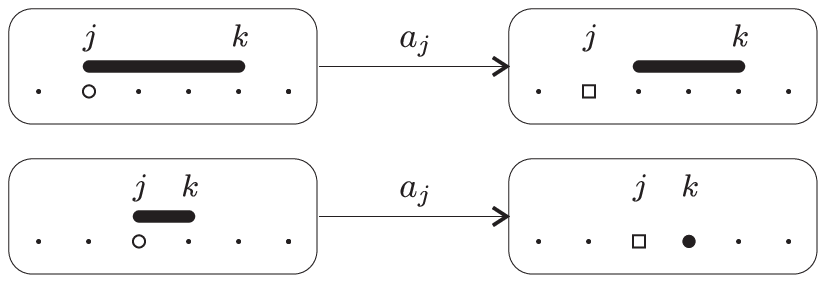}
  \end{center}

 \item $D'$ has black circles at positions $j+2, j+3,\ldots,n$.

   \begin{center}
  \includegraphics{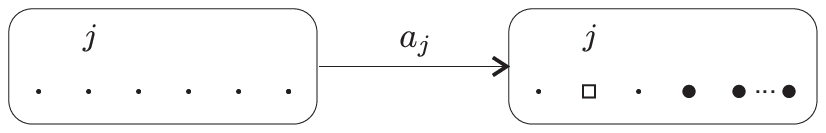}
  \end{center}

 \item $D'$ has a white circle at every position in which there is no square and no black circle.

\end{enumerate}
\end{definition}

We can now show that there are as many sets of minimal forbidden prefixes as segment configurations.

\begin{proposition}\label{P:bijection}
The map $\varphi_n:\: \mathcal F_n \rightarrow \mathcal S_n$ is bijective.
\end{proposition}

\begin{proof}
We have already seen that $\varphi_n$ is a well defined injective map. It just remains to show that $\varphi_n$ is surjective.

Recall that $F_n(1)=\emptyset$ corresponds to the initial state of $\Gamma_n$. \autoref{P:Transitions} explains how to connect states of $\Gamma_n$ by transitions. For every state $F\in \mathcal F_n$, and every permitted letter $a_r$  (that is, $a_r\notin F$), \autoref{P:Transitions} describes the target state $F'$, so we have the transition $F \stackrel{a_r}{\longrightarrow} F'$. If $F=F_n(\beta)$, we have $F'=F_n(\beta a_r)$. All elements of $\mathcal F_n$ are obtained starting at $\emptyset$, and applying a sequence of transitions. We say that all states in $\Gamma_n$ are {\it accessible} from the initial state $\emptyset$.

On the other hand, in \autoref{D:Transitions} we explained a way to connect segment configurations. This defines a directed graph $\widetilde \Gamma_n$, whose vertices are the elements in $\mathcal S_n$, and whose arrows are given by letters $a_r$.


Suppose that we have a transition $(i,j,k,S)\stackrel{a_r}{\longrightarrow} (i',j',k',S')$ in $\widetilde \Gamma_n$. If $(i,j,k,S)=\varphi_n(F_n(\beta))$ for some braid $\beta$, we know, by construction, that $(i',j',k',S')=\varphi_n(F_n(\beta a_r))$. It follows that $\Gamma_n$ can be seen (via $\varphi_n$) as a subgraph of $\widetilde \Gamma_n$. More precisely, $\Gamma_n$ can be seen as the subgraph of $\widetilde \Gamma_n$ formed by the vertices which are accessible from the initial state $\varphi_n(\emptyset) = (n,n,n,\emptyset)$, whose diagram is as follows:
$$
\vcenter{\hbox{\includegraphics{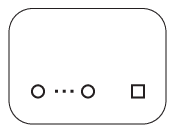}}}
$$

We will then compute the whole graph $\widetilde \Gamma_n$, and we will show that all its vertices are accessible from the initial state, so $\widetilde \Gamma_n$ actually coincides with $\Gamma_n$. This will finish the proof.

Suppose that $n=1$. The only possible segment configuration for $A_1$ is $(1,1,1,\emptyset)$. Then $\widetilde \Gamma_1$ has a single vertex, which by \autoref{P:Transitions} has an arrow labeled $a_1$ starting at ending at it.  Since this vertex is the initial state of $\Gamma_1$, it follows that $\Gamma_1=\widetilde\Gamma_1$, so the result holds for $n=1$.

For every $n\geq 1$, let $\widetilde \Gamma^*_n$  be the subgraph of $\widetilde \Gamma_n$ whose vertices have the form $(1,j,k,S)$. Then we formulate the following

\textbf{Induction hypothesis:}

\begin{enumerate}

\item Every vertex in $\widetilde \Gamma_n$ is accessible from the initial state.

\item Every vertex in $\widetilde \Gamma^*_n$ is accessible from any other vertex in $\widetilde \Gamma^*_n$.

\item Every vertex $(1,j,k,S)\in\widetilde \Gamma^*_n$ with either $j=k$ or $S\neq \emptyset$ is accessible from the vertex $(1,1,1,\emptyset)$ by a path which does not include the label $a_1$.

\end{enumerate}

Observe that these assumptions are clearly true for $n=1$. Let us then suppose that $n>1$, and that the induction hypothesis holds for smaller values of $n$.

In order to describe the whole graph $\widetilde \Gamma_n$, let us distinguish distinct subsets of $\mathcal S_n$:

\begin{itemize}

\item Configurations $(i,j,k,S)$, $i>1$.

\item Configurations $(1,1,k,\emptyset)$.

\item Configurations $(1,j,k,S)$, $1<j$ and no segment $[1,j_1]$ in $S$.

\item Configurations $(1,j,k,\{[1,j]\})$, where $1<j<k$.

\item Configurations $(1,j,k,S)$ with $[1,j_1]\in S$ and $j_1\geq k$.

\end{itemize}

Observe that all these sets of configurations are disjoint, and hence by \autoref{segmentconfig} form a partition of $\mathcal{S}_n$. So let us start our case-by-case analysis.

\underline{Configurations $(i,j,k,S)$ with $i>1$}.

These correspond to the diagrams with a white circle at position 1.

Let us consider the shifting homomorphism $sh: A_{n-1}\rightarrow A_n$ and the shifting map $sh: \mathcal L_{n-1}\rightarrow \mathcal L_n$ determined by $sh(a_t)=a_{t+1}$ for all $t$. We can also define a shifting map $sh: \mathcal S_{n-1}\rightarrow \mathcal S_n$ which increases by 1 all indices in a segment configuration. In terms of diagrams, the shifting just adds a white circle to the left of the diagram. In this picture, we have $sh(C')=C$:
$$
\vcenter{\hbox{\includegraphics{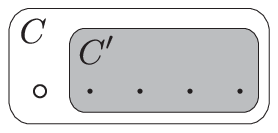}}}
$$

Notice that for every arrow $C_1' \stackrel{a_r}{\longrightarrow} C_2'$ in $\widetilde\Gamma_{n-1}$, there is an arrow $sh(C_1') \stackrel{a_{r+1}}{\longrightarrow} sh(C_2')$ in $\widetilde \Gamma_n$. This covers all arrows starting at vertices of the form $sh(C')$, of index greater than 1. Therefore, all the states $(i,j,k,S)$ with $i>1$ form a complete copy of $\widetilde \Gamma_{n-1} =\Gamma_{n-1}$, but with the labels of the arrows shifted (their indices are increased by one). We will denote this copy $sh(\Gamma_{n-1})$.

Each diagram corresponding to a state of $sh(\Gamma_{n-1})$ has an isolated white circle at position 1. This means that it accepts an outgoing arrow labeled $a_1$. Its target is $(1,1,1,\emptyset)$.  So we have a complete copy of $\Gamma_{n-1}$, with inner arrows labeled $\{a_2,\ldots,a_n\}$, all of whose states have an arrow labeled $a_1$ and pointing to $(1,1,1,\emptyset)$.  We have then described all arrows starting at elements of $sh(\Gamma_{n-1})$.

$$
\vcenter{\hbox{\includegraphics{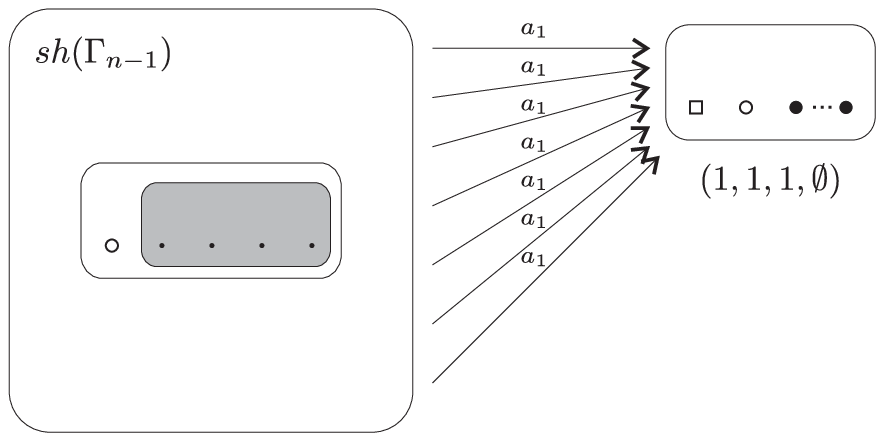}}}
$$

%
%
%

\underline{Configurations $(1,1,k,\emptyset)$}.

These correspond to the diagrams with a square at position 1. We will denote $T_{1,n}=\{(1,1,k,\emptyset);\ k=1,\ldots,n\}\subset \mathcal S_n$ the subset containing these configurations.

The only permitted arrows starting at such states are $a_1$ and $a_2$ (in the latter case, only if $k\neq 2$).

We see that an arrow labeled $a_1$ starting at any $(1,1,k,\emptyset)$ ends at $(1,1,1,\emptyset)$. On the other hand, an arrow labeled $a_2$ starting at $(1,1,k,\emptyset)$ (for $k\neq 2$) ends at $(1,2,k,\emptyset)$. This ends the description of the arrows starting at elements in $T_{1,n}$.

$$
\vcenter{\hbox{\includegraphics{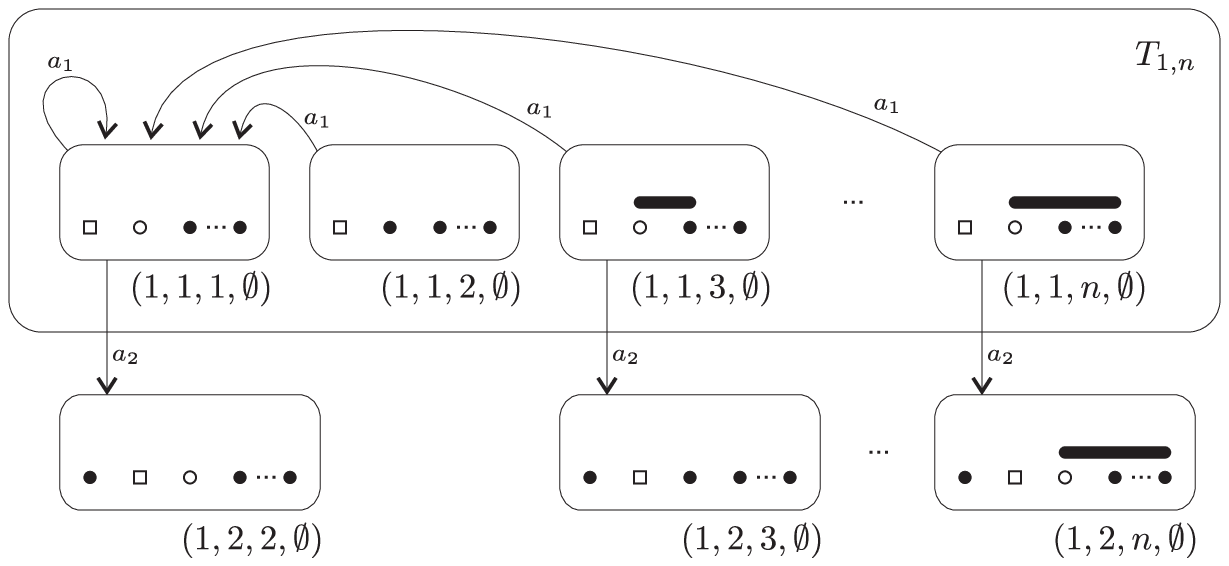}}}
$$

\underline{Configurations $(1,j,k,S)$ with $1<j$ and no segment $[1,j_1]$ in $S$}.

These correspond to the diagrams with a black circle at position 1.

Let $\mathcal S_{n-1}^*$ be the segment configurations in $\mathcal S_{n-1}$ not having an isolated white circle at position 1. That is, those of the form $(1,j',k',S')$. Recall that $\widetilde \Gamma^*_{n-1}$ is the subgraph of $\widetilde \Gamma_{n-1}$ whose vertices belong to $\mathcal S_{n-1}^*$. We know by induction that we can identify $\widetilde \Gamma_{n-1}$ and $\Gamma_{n-1}$, so we will just denote $\Gamma^*_{n-1}=\widetilde \Gamma^*_{n-1}$.

Given $C'=(1,j',k',S')\in \mathcal S_{n-1}^*$, we have $sh(C')=(2,j,k,S)$, and we define $sh_{\bullet}(C')=(1,j,k,S)$. Every configuration whose diagram has a black circle at position 1 can be obtained in this way. Hence, the subgraph of $\widetilde \Gamma_{n-1}$ containing these vertices will be denoted $sh_{\bullet}(\Gamma^*_{n-1})$.

The graph $sh_{\bullet}(\Gamma^*_{n-1})$ is quite similar to $\Gamma^*_{n-1}$: It has the same number of vertices and, for every arrow $C_1' \stackrel{a_r}{\longrightarrow} C_2'$ in $\Gamma^*_{n-1}$, with $r\geq 2$, there is an arrow $sh_{\bullet}(C_1') \stackrel{a_{r+1}}{\longrightarrow} sh_{\bullet}(C_2')$ in $\widetilde \Gamma_n$, as the arrow $a_{r+1}$ does not affect the black circle at position 1. The difference comes from the arrows labeled $a_2$ starting at states of $sh_{\bullet}(\Gamma^*_{n-1})$.

Suppose that a state from $sh_{\bullet}(\Gamma^*_{n-1})$ admits an outgoing arrow labeled $a_2$. In this case, in the diagram corresponding to this state, there is either a square at position 2 or a segment starting at position 2 and ending at some position $k\in \{3,\ldots,n\}$. And we know that there is a black circle at position 1. Therefore, at the target of the arrow $a_2$ we have the segment configuration $(1,2,k,\{[1,2]\})$, for some $k=2,\ldots,n$. Hence, there are exactly $n-1$ states which receive the outgoing arrows $a_2$ from $sh_{\bullet}(\Gamma^*_{n-1})$.

Since every state in $sh_{\bullet}(\Gamma^*_{n-1})$ has a black circle at position 1, the letter $a_1$ is forbidden after any such state, so there is no arrow labeled $a_1$ starting at any such state.  We have then described all arrows starting at elements in $sh_{\bullet}(\Gamma^*_{n-1})$.
$$
\vcenter{\hbox{\includegraphics{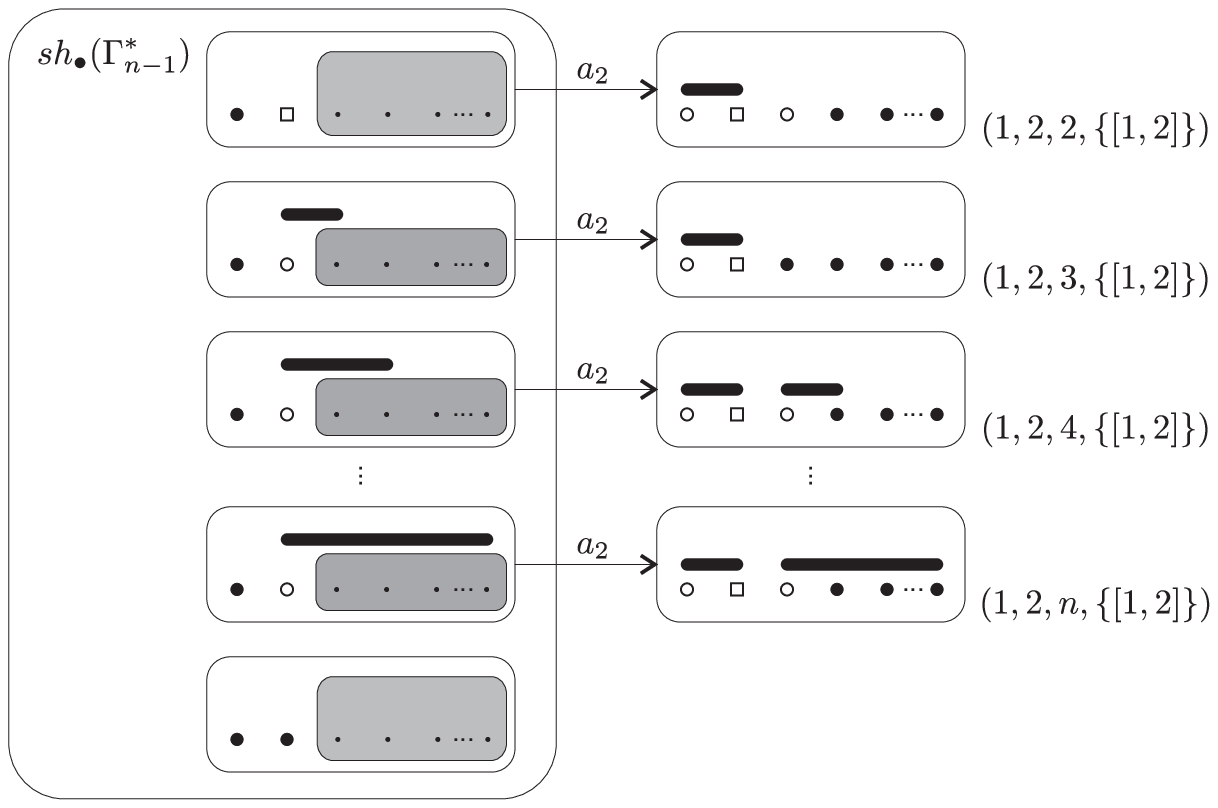}}}
$$

\underline{The states $(1,j,k,\{[1,j]\})$, where $1<j<k$}.

These correspond to diagrams with a segment $[1,j]$ ($j$ is the position of the square), and either a black circle at $j+1$ or a segment $[j+1,k]$. For every $j=2,\ldots,n-1$, we will denote $T_{j,n}=\{(1,j,k,\{[1,j]\});\ k=j+1,\ldots,n\}\subset \mathcal S_n$ the subset containing such configurations.

The arrows starting at such a state are labeled $a_1$, $a_j$ and $a_{j+1}$, the latter case only if $k>j+1$. In the latter case, we have the arrow:
$$
   (1,j,k,\{[1,j]\}) \stackrel{a_{j+1}}{\longrightarrow} (1,j+1,k,\{[1,j+1]\}).
$$
So, for every $k>1$ we have a chain:
$$
  (1,2,k,\{[1,2]\}) \stackrel{a_{3}}{\longrightarrow} (1,3,k,\{[1,3]\}) \stackrel{a_{4}}{\longrightarrow} \cdots  \stackrel{a_{k-1}}{\longrightarrow} (1,k-1,k,\{[1,k-1]\}).
$$

Now, from every state $(1,j,k,\{[1,j]\})$ with $1<j<k$, there is an arrow labeled $a_1$ going to $(1,1,j,\emptyset)$.

Finally, from every state $(1,j,k,\{[1,j]\})$ with $1<j<k$, there is an arrow labeled $a_j$ going to $(1,j,j,\{[1,j],[j-1,j]\})$.

$$
\vcenter{\hbox{\includegraphics{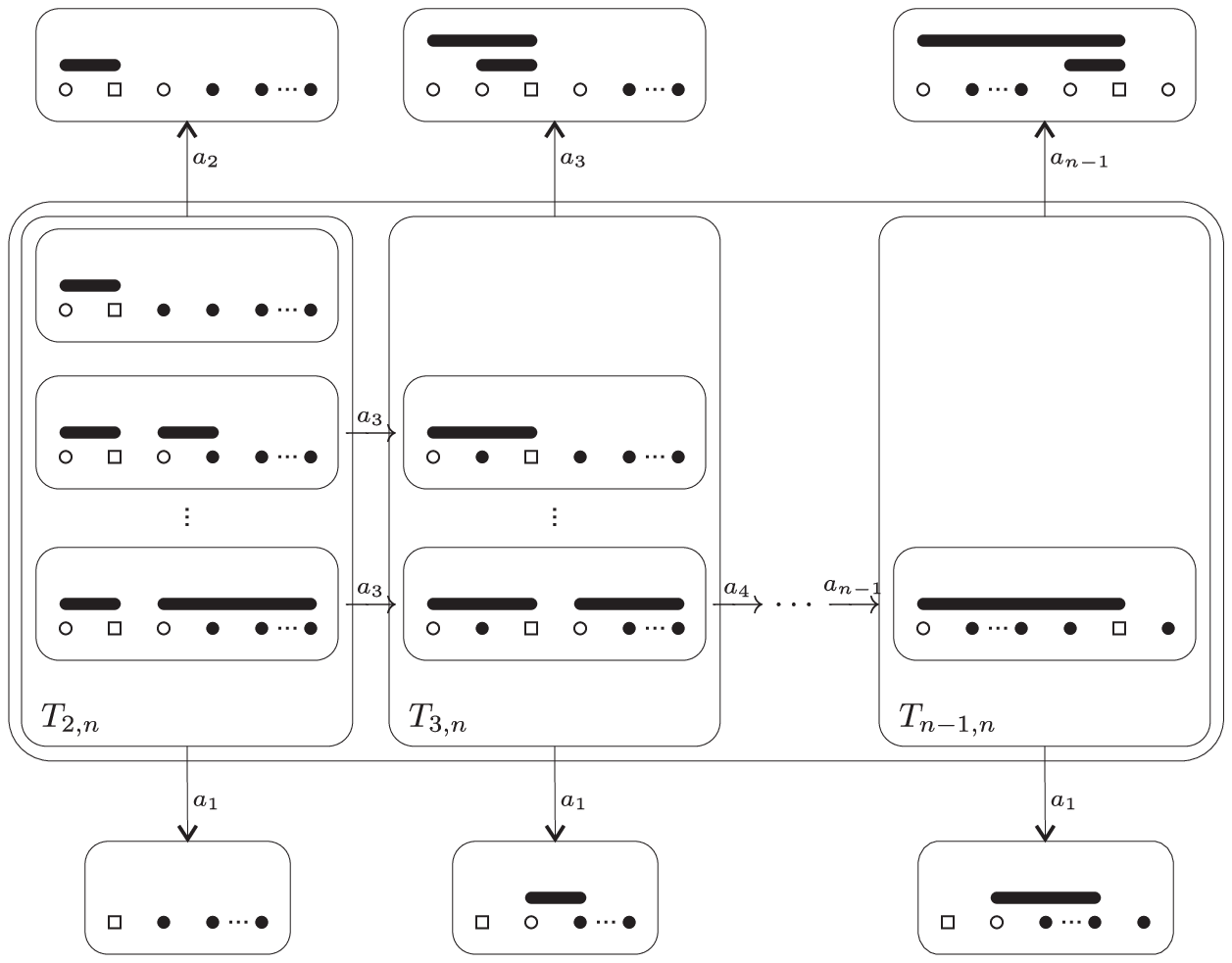}}}
$$

\underline{The states $(1,j,k,S)$ with $[1,j_1]\in S$ and $j_1\geq k$}.

These are the remaining configurations in $\mathcal S_n$. They have a segment from 1 to $j_1$. To the right of the segment they have only black circles, unless $j=j_1<n$: in that case they have a white circle at position $j_1+1$.

The permitted labels for arrows starting at such configurations are $a_1$ and maybe some letters in $\{a_2,\ldots,a_{j+1}\}$.

The target of an arrow labeled $a_1$ is $(1,1,j_1,\emptyset)$, in any case.

Fix $j_1\in \{2,\ldots,n\}$. If we remove the segment $[1,j_1]$ from the configuration $C_1=(1,j,k,S)$, we obtain the configuration $C_1^*=(2,j,k,S^*)$, which only involves the numbers $\{2,\ldots,j_1\}$. We can consider this configuration as a state of $sh(\Gamma^*_{j_1-1})$. Every arrow $C_1^* \stackrel{a_r}{\longrightarrow} C_2^*$ in $sh(\Gamma^*_{j_1-1})$, with $r\in \{2,\ldots,j_1\}$, corresponds to an arrow $C_1 \stackrel{a_r}{\longrightarrow} C_2$, where $C_i$ is obtained from $C_i^*$ by adding the segment $[1,j_1]$. This means that the elements containing the segment $[1,j_1]$ form a complete copy of the graph $sh(\Gamma^*_{j_1-1})$. We will denote this subgraph $\overline{sh(\Gamma^*_{j_1-1})}$.

There is still another possible arrow starting at such kind of configuration, but it can only happen if $j=k=j_1<n$. In this case, $a_{j_1+1}=a_{j+1}$ is a permitted arrow. Also, in this case, all segments in $S$ must end at $j$, so we have $S=\{[1,j],[i_2,j],\ldots,[i_t,j]\}$. Then, denoting $S'=\{[1,j+1],[i_2,j+1],\ldots,[i_t,j+1]\}$, the target of the arrow $a_{j+1}$ starting at $(1,j,k,S)=(1,j,j,S)$ is the configuration $(1,j+1,j+1,S')$.

$$
\vcenter{\hbox{\includegraphics[scale=0.9]{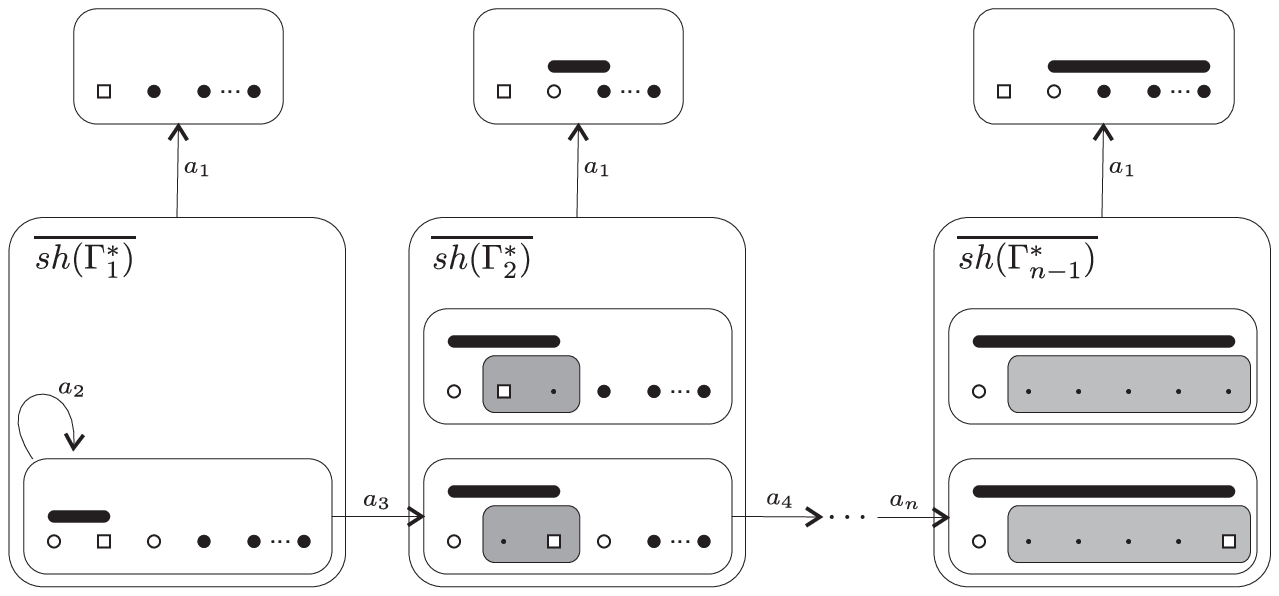}}}
$$

There are no other possible arrows starting at these configurations, so we have described all possible transitions between the distinct configurations in $\mathcal S_n$.

Joining all the pieces, we have a complete picture of the graph $\widetilde \Gamma_n$, in \autoref{F:Gamma_n}.

\begin{sidewaysfigure}
\includegraphics[scale=0.5]{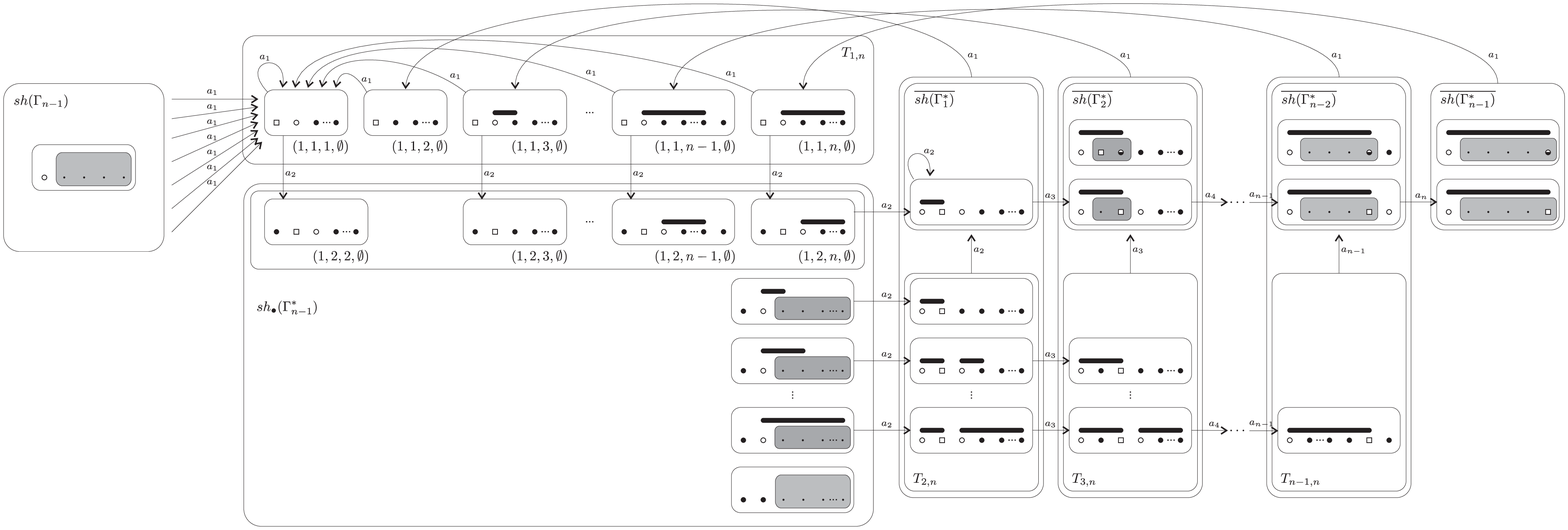}
\caption{Structure of the automaton $\Gamma_n$}
\label{F:Gamma_n}
\end{sidewaysfigure}

Now let us see that $\widetilde \Gamma_n= \Gamma_n$, and that the induction hypothesis stated at the the beginning of this proof is satisfied. We first need to see that every vertex of $\widetilde \Gamma_n$ is accessible from the initial vertex $(n,n,n,\emptyset)$.

First, we see that $(n,n,n,\emptyset)\in sh(\Gamma_{n-1})$. Moreover, it is the shift of $(n-1,n-1,n-1,\emptyset)$, which is the initial state of $\Gamma_{n-1}$. By induction hypothesis we know that every vertex of $\Gamma_{n-1}$ is accessible from the initial state, hence every vertex of $sh(\Gamma_{n-1})$ is accessible from $(n,n,n,\emptyset)$.

Recall that every vertex in $sh(\Gamma_{n-1})$ has an arrow pointing to $(1,1,1,\emptyset)\in \mathcal S_n$. Hence, this vertex is accessible. Now we can apply the transition $a_2$, to obtain $(1,2,2,\emptyset)$, which is a vertex in $sh_{\bullet}(\Gamma^*_{n-1})$.

By induction hypothesis, all vertices in $\Gamma^*_{n-1}$ of the form $(1,j,k,S)$, with either $j=k$ or $S\neq \emptyset$, are accessible from $(1,1,1,\emptyset)$ by a path not involving the label $a_1$. Therefore, as the arrows in $sh_{\bullet}(\Gamma^*_{n-1})$ with labels in $\{a_3,\ldots,a_n\}$ coincide (with shifted indices) with the arrows in $\Gamma^*_{n-1}$ with labels in $\{a_2,\ldots,a_{n-1}\}$, we obtain that every vertex in $sh_{\bullet}(\Gamma^*_{n-1})$ not having the form $(1,j,k,\emptyset)$ ($j<k$), is accessible from $(1,2,2,\emptyset)=sh_{\bullet}((1,1,1,\emptyset))$.

Now we can access all vertices in $T_{j,n}$ for $j=2,\ldots,n-1$, starting at suitable (accessible) vertices in $sh_{\bullet}(\Gamma^*_{n-1})$, for instance $(1,j,j,\{[2,j]\})$ for $j=2,\ldots, n$, and applying the transitions $a_2, a_3,\ldots, a_{n-1}$.

Given $j\in \{2,\ldots,n-1\}$, we can apply the transition $a_j$ to any vertex in $T_{j,n}$, to obtain the vertex $(1,j,j,\{[1,j]\})$.  This is, hence, an accessible vertex in $\overline{sh(\Gamma^*_{j-1}})$. Since $\overline{sh(\Gamma^*_{j-1})}$ is a complete copy of $\Gamma^*_{j-1}$, with the indices shifted, it follows that all vertices in $\overline{sh(\Gamma^*_{j-1})}$ are accessible from any other. This holds for $j=2,\ldots,n-1$.

Now we can access $\overline{sh(\Gamma^*_{n-1})}$ from $\overline{sh(\Gamma^*_{n-2})}$, since there is an arrow labeled $a_n$, for instance, from $(1,n-1,n-1,\{[1,n-1]\})$ to $(1,n,n,\{[1,n]\})$. It follows that all vertices in $\overline{sh(\Gamma^*_{n-1})}$ are also accessible.

Notice that we have already accessed, from $(1,1,1,\emptyset)$, all elements in $\widetilde \Gamma_n$ of the form $(1,j,k,S)$ except those of the form $(1,j,k,\emptyset)$ with $j<k$, by a path which does not involve the label $a_1$.

Finally, from any vertex in $\overline{sh(\Gamma^*_{k-1})}$ (for $k=2,\ldots,n$), there is a transition $a_1$ to the vertex $(1,1,k,\emptyset)$. And from this one, we have a sequence of transitions $a_2,a_3,\dots,a_{k-1}$, which lead successively to $(1,2,k,\emptyset)$, $(1,3,k,\emptyset),\ldots, (1,k-1,k,\emptyset)$. This covers all the remaining vertices of $\widetilde \Gamma_n$.

Therefore, all vertices of $\widetilde \Gamma_n$ are accessible from the initial vertex $(n,n,n,\emptyset)$. This implies that $\widetilde \Gamma_n = \Gamma_n$. Moreover, it follows by construction that all elements in $\Gamma_n$ of the form $(1,j,k,S)$ except those of the form $(1,j,k,\emptyset)$ with $j<k$, are accessible from $(1,1,1,\emptyset)$ by a path which does not involve the label $a_1$.  Finally, we can easily check that from every vertex in $\Gamma^*_n$ we can access $(1,1,1,\emptyset)$ (following the transitions explained in the paragraphs above, and noticing that $(1,1,1,\emptyset)$ can be accessed from $(1,1,k,\emptyset)$ by a transition $a_1$). Therefore, we can go from any vertex in $\Gamma^*_n$ to any other, passing through $(1,1,1,\emptyset)$.

This shows the induction step, and finishes the proof. So we have described the whole automaton $\Gamma_n$.
\end{proof}

\section{Size of the automaton}\label{S:size}

One nice consequence of \autoref{P:bijection} is that we can count the number of states of the automaton $\Gamma_n$, as they are in bijection with the number of segments configurations for $A_n$.

Instead of counting directly the number of segment configurations, we will derive a recurrence relation between the number of segment configurations for $A_n$ and the number of segment configurations for $A_m$, with $m<n$. From that recurrence relation we will obtain the final formula for the number of states of $\Gamma_n$, which surprisingly depends on the Fibonacci numbers.

\begin{proposition}\label{P:recurrence_relation}
For $n\geq 0$, let $s_n$ be the number of states in $\Gamma_n$, which coincides with the number of segment configurations for $A_n$ (we are considering $\Gamma_0=\emptyset$). We have $s_0=0$, $s_1=1$, and the following recurrence relation:
$$
    s_n=3\:s_{n-1}-s_{n-2}+{n\choose 2}+1.
$$
\end{proposition}

\begin{proof}
By definition, $s_0=0$. If $n=1$, the only segment configuration is $(1,1,1,\emptyset)$, so $s_1=1$. Let us assume that $n>1$. In order to count the elements in $\mathcal S_n$, we will consider the parts of the automaton $\Gamma_n$ that were described in \autoref{T:automaton}.

First, the vertices of $sh(\Gamma_{n-1})$ are in bijection with the vertices in $\Gamma_{n-1}$, so there are $s_{n-1}$.

Next, the vertices in $sh_{\bullet}(\Gamma^*_{n-1})$ are in bijection with the vertices in $\Gamma^*_{n-1}$, which are the vertices in $\Gamma_{n-1}$ which do not belong to $\Gamma_{n-2}$. There are exactly $s_{n-1}-s_{n-2}$.

Now consider the vertices in $T_{1,n}, T_{2,n},\ldots, T_{n-1,n}$. The number of vertices in $T_{j,n}$ is $n-j$, except in the case of $T_{1,n}$, which has $n$ vertices. Therefore, there are $n+(n-2)+(n-3)+\cdots +1$ vertices, which is precisely ${n \choose 2}+1$.

Finally, the set of vertices in $\overline{sh(\Gamma^*_{j})}$, for $j\in\{1,\ldots,n-1\}$, is in bijection with $\Gamma^*_j$, that is, it has $s_j-s_{j-1}$ vertices. Adding up the number of vertices of these sets, for $j=1,\ldots,n-1$, we have:
$$
  (s_{n-1}-s_{n-2})+(s_{n-2}+s_{n-3})+\cdots+(s_1-s_0)= s_{n-1}-s_0=s_{n-1}.
$$
So there are $s_{n-1}$ such vertices.

The disjoint union of all considered sets of vertices forms $\mathcal S_n$, so the number of elements in $\mathcal S_n$ is exactly:
\begin{eqnarray*}
   s_n & = & s_{n-1}+(s_{n-1}-s_{n-2})+{n\choose 2}+1+(s_{n-1}+1) \\
     & = & 3\:s_{n-1}-s_{n-2}+{n\choose 2}+1.
\end{eqnarray*}
\end{proof}

\begin{theorem}\label{T:numer_of_states}
Let $F_n$ be the $n$-th Fibonacci number ($F_0=0$, $F_1=1$, $F_{i+1}=F_i+F_{i-1}$). Let $s_n$ be the number of states of the Finite State Automaton $\Gamma_n$ accepting the language of maximal (or minimal) lexicographic representatives of elements in the positive braid monoid $A_n$. Then
$$
   s_n=\sum_{i=1}^n{\left({n+1-i \choose 2}+1\right)\cdot F_{2i}}
$$
\end{theorem}

\begin{proof}
For $n=0$ and $n=1$ the result is clear, as $s_0=0$ and $s_1=1\cdot F_2 = 1$. Let us suppose that $n>1$ and that the result holds for smaller values of $n$. We know by \autoref{P:recurrence_relation} that
$$
  s_n  = 3\:s_{n-1}-s_{n-2}+{n\choose 2}+1.
$$
By induction hypothesis:
$$
  s_n = 3\: \left(\sum_{i=1}^{n-1}{\left({n-i \choose 2}+1\right)\cdot F_{2i}}\right)- \left(\sum_{i=1}^{n-2}{\left({n-1-i \choose 2}+1\right)\cdot F_{2i}}\right) + {n\choose 2} + 1
$$
$$
= 3\: \left(\sum_{i=2}^{n}{\left({n+1-i \choose 2}+1\right)\cdot F_{2i-2}}\right)- \left(\sum_{i=3}^{n}{\left({n+1-i \choose 2}+1\right)\cdot F_{2i-4}}\right) + {n\choose 2} + 1
$$
$$
= \left(\sum_{i=2}^{n}{\left({n+1-i \choose 2}+1\right)\cdot (3\: F_{2i-2}-F_{2i-4})}\right)+ \left({n\choose 2} + 1\right)\cdot F_2.
$$
Now we just need to see that
$$
F_{2i}+F_{2i-4}=F_{2i-1}+F_{2i-2}+ F_{2i-4} =2F_{2i-2} +F_{2i-3}+F_{2i-4}= 3 F_{2i-2}.
$$
Hence $3\: F_{2i-2}-F_{2i-4}= F_{2i}$, and we finally have:
$$
s_n = \left(\sum_{i=2}^{n}{\left({n+1-i \choose 2}+1\right)\cdot F_{2i}}\right)+ \left({n\choose 2} + 1\right)\cdot F_2
$$
$$
  = \sum_{i=1}^{n}{\left({n+1-i \choose 2}+1\right)\cdot F_{2i}},
$$
as we wanted to show.
\end{proof}

We finish the section by explicitly computing the first terms of the sequence $\{s_n\}$.
$$
\begin{array}{llllr}
   s_1 & = &  \phantom{1}1\cdot 1 & = & 1 \\
   s_2 & = &  \phantom{1}2\cdot 1 + 1\cdot 3 & = & 5 \\
   s_3 & = &  \phantom{1}4\cdot 1 + 2\cdot 3 + 1\cdot 8  & = & 18 \\
   s_4 & = &  \phantom{1}7\cdot 1 + 4\cdot 3 + 2\cdot 8 + 1\cdot 21 & = & 56 \\
   s_5 & = & 11\cdot 1 + 7\cdot 3 + 4\cdot 8 + 2\cdot 21 + 1\cdot 55 & = & 161 \\
\end{array}
$$

These numbers were computed in~\cite{GG} for $n=1,\ldots, 19$, without the knowledge of the above formula. Just to understand the size of the numbers involved, one has $s_{19}=126\: 491\: 780$.

\section{Asymptotic properties}\label{S:asymptotic}

We will now use our knowledge of the structure of the DFSA $\Gamma_n$, to determine asymptotic properties of braid monoids $A_n$. A good reference to the linear algebra used and Perron-Frobenius theory in this section is the book \cite{Sen81}.

Consider the {\it incidence matrix} $M_n$ of the DFSA $\Gamma_n$. This is a $s_n\times s_n$ matrix (where $s_n$ is the number of states of $\Gamma_n$), where each row (resp. each column) corresponds to a state of $\Gamma_n$. The entry $m_{i,j}$ of $M_n$ equals 1 if there is a transition from the state $i$ to the state $j$ in $\Gamma_n$, and equals 0 otherwise.

For instance, in \autoref{F:Gamma_2} we can see the DFSA $\Gamma_2$, whose incidence matrix is the following:
$$
    M_2=\begin{pmatrix}
      1 & 1 & 0 & 0 & 0 \\
      0 & 1 & 0 & 1 & 0 \\
      0 & 1 & 0 & 0 & 0 \\
      0 & 0 & 0 & 0 & 1 \\
      0 & 0 & 1 & 0 & 1
    \end{pmatrix}
$$

\begin{figure}
\begin{center}
\includegraphics{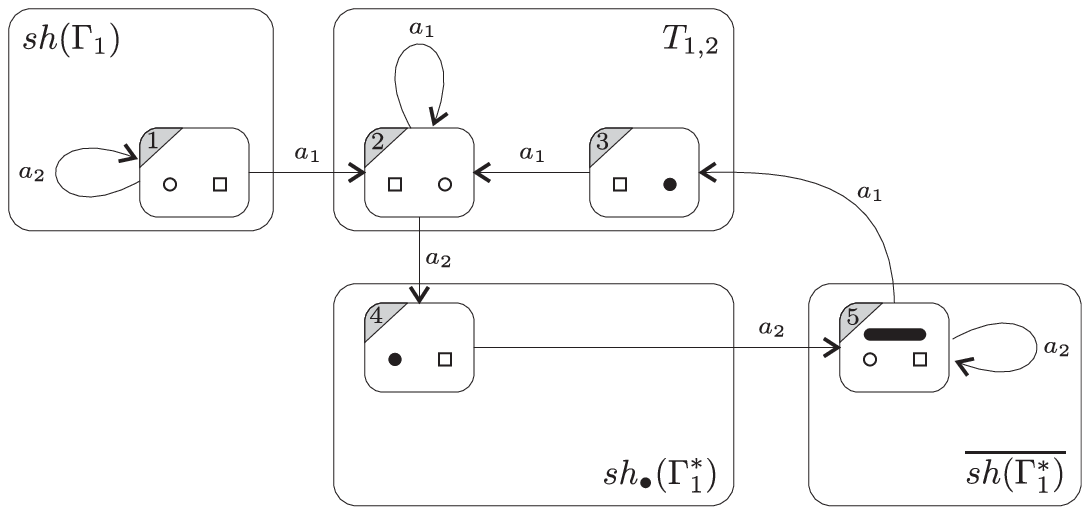}
\end{center}
\caption{The DFSA $\Gamma_2$.}
\label{F:Gamma_2}
\end{figure}

Taking powers of the matrix $M_n$, one obtains information about the number of paths from one state to another. The entry $(i,j)$ of $(M_n)^k$ is precisely the number of paths of length $k$ from the state $i$ to the state $j$. Recall that the elements in $\mathcal L_n$ (which are in bijection with the elements in $A_n$) correspond to the paths in $\Gamma_n$ starting at the initial state (state 1). Therefore, the entries of the first row of $(M_n)^k$ determine the proportion of elements of length $k$ in $\mathcal L_n$ ending at each state.

For instance, if we compute the 50th power of the matrix $M_2$ above, we obtain:
$$
(M_2)^{50}=\begin{pmatrix}
  1 & 16475640050 & 10182505536 & 10182505537 & 16475640048 \\
  0 & 10182505537 & 6293134512 & 6293134513 & 10182505537 \\
  0 & 6293134513 & 3889371024 & 3889371025 & 6293134512 \\
  0 & 6293134512 & 3889371025 & 3889371024 & 6293134513 \\
  0 & 10182505537 & 6293134513 & 6293134512 & 10182505537
\end{pmatrix}
$$

It seems clear that the probability of ending at the initial state tends to 0 as $k$ tends to infinity, since there is only one possible path ($a_2a_2a_2\cdots a_2$) ending at the initial state, and this is negligible compared to the number of paths of length $k$ ending at any other state. We slightly abuse language, by calling ``probability of ending at a state" the proportion of accepted paths ending at that state as $k$ tends to infinity. Following the analogy, the nomenclature of the next paragraph is inspired by the terminology of Markov chains.

In the previous section we have shown two important properties of $\Gamma_n$. First, the states in $sh(\Gamma_{n-1})$ are {\it transient}: once the process reaches the subgraph $\Gamma^*_n$, it will never go back to $sh(\Gamma_{n-1})$. This is why the probability of ending at a state in $sh(\Gamma_{n-1})$ tends to 0 as $k$ tends to infinity. On the other hand, the states in $\Gamma^*_n$ (all the remaining states in $\Gamma_n$) are {\it recurrent}: for every ordered pair $(v_1,v_2)$ of states in $\Gamma^*_n$, there is always a path going from $v_1$ to $v_2$. This means that the probability of ending at such a state will tend to a positive number.

It is then clear that, if we want to study the proportion of accepted paths ending at a given state (when $k$ tends to infinity), we can restrict our attention to the recurrent states, and forget the transient ones. We will then consider the incidence matrix $R_n$ whose rows (resp. columns) correspond to recurrent states. In the above example ($n=2$), we just need to remove the first row and column, and we obtain the matrix:
$$
   R_2= \begin{pmatrix}
     1 & 0 & 1 & 0 \\
     1 & 0 & 0 & 0 \\
     0 & 0 & 0 & 1 \\
     0 & 1 & 0 & 1
   \end{pmatrix}
$$
Since $R_n$ is the incidence matrix of the recurrent states, all entries of some power of $R_n$ (and all successive powers after it) will be positive integers. This means that $R_n$ is a {\it primitive} matrix.  The Perron-Frobenius eigenvalue of $R_n$ is the eigenvalue $\lambda_n$ of biggest modulus (which turns out to be a positive real number), and the Perron-Frobenius eigenvector is the only row vector $v$ (up to scalar multiplication) with positive entries, such that $v R_n=\lambda_n v$.  The proportions between the entries of this vector are precisely the limits, when $k$ tends to infinity, of the proportions between the entries of any row of $(R_n)^k$. In our case, they are the proportions of `very long' elements in $\mathcal L_n$ ending at each recurrent state.

In the particular case $n=2$, the Perron-Frobenius eigenvalue is the golden ratio $\varphi=\frac{1+\sqrt{5}}{2}=1.618033\cdots$, and the Perron-Frobenius eigenvector is $v= (\varphi, 1,1, \varphi)$.  This means that, among the elements in $\mathcal L_2$ of length $k>>1$, it is $\varphi$ times more likely to end at states 2 or 5 than to end at states 3 or 4.  We can already guess this proportion in the matrix $(M_2)^{50}$ showed above.

We also see in \autoref{F:Gamma_2} that the final letter of states 2 and 3 is $a_1$, and the final letter of states 4 and 5 is $a_2$. This implies that, when $k$ tends to infinity, the proportion of lexicographic representatives ending with $a_1$ equals $1/2$: there are as many representatives finishing with $a_1$ as representatives finishing with $a_2$.

One could think that, when $n>2$, the proportion of long representatives finishing with any given $a_j$ is not far from $1/n$. But this is not the case. The lexicographic representatives finishing with $a_1$ become more abundant. Actually, we will show that this proportion is bounded below by $\frac{1}{8}$, no matter how big is $n$.

It is well known that the eigenvalue $\lambda_n$ associated to $\Gamma_n$ measures the growth of the entries of $(R_n)^k$. Hence, $\lambda_n$ is the growth rate of the monoid $A_n$ (with respect to the standard generators $a_1,\ldots,a_n$). Fortunately, it is already known how these growth rates relate to each other, for distinct values of $n$:

\begin{theorem}\label{T:limit<4}{\rm \cite[Proposition 7.98]{Jug16}} The sequence $\{\lambda_n\}_{n\geq 1}$ of Perron-Frobenius eigenvalues (growth rates) for the braid monoids $A_n$, is a strictly increasing sequence of positive real numbers. Its limit $\lambda_{\infty}$ as $n$ tends to infinity satisfies $2.5 < \lambda_{\infty} < 4$.
\end{theorem}

We remark that, as we have seen, $\Gamma_{n-1}^*$ is a proper complete subgraph of $\Gamma_n^*$. This implies immediately that $\lambda_{n-1}<\lambda_n$. So the study of the automaton $\Gamma_n$ in this paper also shows that the sequence $\{\lambda_n\}$ is increasing.

We are going to use the following technical lemma, which is well-known. Recall that the \emph{spectral radius} of a square matrix is the largest absolute value of its eigenvalues.

\begin{lemma}\label{L:positive_matrix}
Let $R$ be a square matrix whose elements are nonnegative, and let $\lambda$ be a positive real number, strictly bigger than the spectral radius of $R$. Then all entries of $(\lambda I-R)^{-1}$ are nonnegative. If $R$ is primitive, then all entries of $(\lambda I-R)^{-1}$ are positive.
\end{lemma}

\begin{proof}
Since $\lambda$ is bigger than the spectral radius of $R$, the spectral radius of $\lambda^{-1}R$ is smaller than 1. In this case, we have the convergent sum:
$$
  (I-\lambda^{-1}R)^{-1} = I + (\lambda^{-1}R)+(\lambda^{-1}R)^2+(\lambda^{-1}R)^3+\cdots
$$
Multiplying by $\lambda^{-1}$, we have:
$$
  (\lambda I-R)^{-1} = \lambda^{-1}I + \lambda^{-2}R+\lambda^{-3}R^2+\lambda^{-4}R^3+\cdots
$$
Since all entries of $R$ are nonnegative, the above (convergent) infinite sum is made of matrices with nonnegative entries, hence all entries of $(\lambda I-R)^{-1}$ are nonnegative. If furthermore $R$ is primitive, then some power of $R$ has all entries positive, so all entries of $(\lambda I-R)^{-1}$ are positive.
\end{proof}

In order to study the Perron-Frobenius eigenvector $v$ of the incidence matrix $R_n$, we will interpret it as follows: The vector $v=(v_1,v_2,\ldots,v_{m})$ is an assignment of a positive real number $v_i$ to each recurrent state $S_i$ of $\Gamma^*_n$. Recall that $v R_n = \lambda_n v$. This means that, for every state $S_j$, if we denote $\mbox{Pred}(S_j)=\{S_{i_1},\ldots,S_{i_k}\}$ the states of $\Gamma^*_n$ for which there is a transition $S_{i_r}\longrightarrow S_j$, which are called the {\it direct predecessors} of $S_j$, then we must have
$$
   v_{i_1}+\cdots +v_{i_k}= \lambda_n v_j.
$$
Moreover, the vector $v$ satisfying this property is unique, up to scalar multiplication. So we can normalize it  in such a way that $v_1+\cdots +v_m=1$.  In this way, the value $v_i$ is the probability that a `very long' lexicographic representative ends at the state $S_i$.

Let us introduce some notation. Suppose that $\Gamma$ is a DFSA all of whose states are recurrent. Hence its incidence matrix $R$ is primitive and has a Perron-Frobenius eigenvalue $\lambda$. Let $\mathcal I=\{i_1,\ldots,i_t\}$ be a proper subset of vertices of $\Gamma$, and let $\Gamma_{\mathcal I}$ be the complete subgraph of $\Gamma$ determined by these vertices. Let $R_{\mathcal I}$ be the incidence matrix of $\Gamma_{\mathcal I}$. As $R_{\mathcal I}$ can be seen as a submatrix of $R$, it is known that its spectral radius is smaller than  $\lambda$.

Let $v=(v_1,\ldots,v_m)$ be the normalized Perron-Frobenius eigenvector associated to $\Gamma$. Given a state $S_{i}$ of $\Gamma_{\mathcal I}$, we will be interested in the direct predecessors of $S_{i}$ which do not belong to $\Gamma_{\mathcal I}$, as they are the states which, informally speaking, `introduce probability' in $\Gamma_{\mathcal I}$. So we denote, for every $i\in \mathcal I$:
$$
   P_{i,\mathcal I}=\mbox{Pred}(S_{i})\setminus \Gamma_{\mathcal I}.
$$
We will measure the `incoming probability' at the state $S_{i}$ with the following number:
$$
   p_{i,\mathcal I}=\sum_{j\in P_{i,\mathcal I}}v_j.
$$
And the total `incoming probability' of the subset $\mathcal I$ will be determined by a vector
$$
  p_{\mathcal I}= (p_{i_1,\mathcal I},\; p_{i_2,\mathcal I},\; \ldots,\; p_{i_t,\mathcal I}).
$$
As usual, we will compare vectors $u=(u_1,\ldots,u_t)$ and $w=(w_1,\ldots,w_t)$ of the same length, by saying that $u\leq w$ if $u_i\leq w_i$ for $i=1,\ldots,t$.

Let us see that, under some conditions, the coordinates of $p_{\mathcal I}$ may give some information about the coordinates of the eigenvector $v$ inside $\mathcal I$.

\begin{proposition}\label{P:subgraphs_comparison}
Let $\Gamma$ be a DFSA whose vertices are all recurrent, so its incident matrix $R$ is primitive. Let $\lambda$ be its Perron-Frobenius eigenvalue, and $v=(v_1,\ldots,v_m)$ its normalized Perron-Frobenius eigenvector. Let $\mathcal I=\{i_1,\ldots,i_t\}$ and $\mathcal J=\{j_1,\ldots,j_t\}$ be two proper subsets of vertices of $\Gamma$, such that the subgraphs $\Gamma_{\mathcal I}$ and $\Gamma_{\mathcal J}$ are isomorphic (by the map sending $i_k$ to $j_k$ for $k=1,\ldots,t$). If $p_{\mathcal I}\geq \alpha\: p_{\mathcal J}$ for some positive real number $\alpha$, then:
$$
   (v_{i_1},\ldots,v_{i_t})\geq \alpha\: (v_{j_1},\ldots,v_{j_t}).
$$
Moreover, if the incidence matrix for $\mathcal I$ is primitive, and $p_{\mathcal I}> \alpha\: p_{\mathcal J}$, then
$$
   (v_{i_1},\ldots,v_{i_t})> \alpha\: (v_{j_1},\ldots,v_{j_t}).
$$
\end{proposition}

\begin{proof}
We know that, for every $i\in \mathcal I$, the sum of all $v_k$ with $S_k\in \mbox{Pred}(S_{i})$ equals $\lambda v_{i}$. This sum can be split in two: on one side we can take the states which belong to $\Gamma_{\mathcal I}$, which are determined by the matrix $R_{\mathcal I}$, and on the other side we can take the states which do not belong to $\Gamma_{\mathcal I}$, which determine the `incoming probability' $p_{i,\mathcal I}$. Collecting this information for all $i\in \mathcal I$, we have the following:
$$
   (v_{i_1}\cdots v_{i_t})R_{\mathcal I} + p_{\mathcal I} = \lambda (v_{i_1}\cdots v_{i_t}).
$$
Hence:
$$
  (v_{i_1}\cdots v_{i_t})(\lambda I -R_{\mathcal I})=  p_{\mathcal I}.
$$
Since $\lambda$ is bigger than the greatest eigenvalue of $R_{\mathcal I}$, the matrix $(\lambda I -R_{\mathcal I})$ is invertible. So we have:
$$
  (v_{i_1}\cdots v_{i_t})=  p_{\mathcal I} \: (\lambda I -R_{\mathcal I})^{-1}
$$
In the same way, we have:
$$
  (v_{j_1}\cdots v_{j_t})=  p_{\mathcal J} \: (\lambda I -R_{\mathcal J})^{-1}
$$
Now we just need to recall that $R_{\mathcal I}=R_{\mathcal J}$, and to notice that, as $\lambda$ is greater than the spectral radius of $R_{\mathcal I}$, \autoref{L:positive_matrix} ensures that all entries of $(\lambda I -R_{\mathcal I})^{-1} $ are nonnegative. Therefore:
$$
  (v_{i_1}\cdots v_{i_t})=  p_{\mathcal I} \: (\lambda I -R_{\mathcal I})^{-1} \geq \alpha \:  p_{\mathcal J} \: (\lambda I -R_{\mathcal I})^{-1}
$$
$$
  = \alpha \:  p_{\mathcal J} \: (\lambda I -R_{\mathcal J})^{-1} = \alpha\: (v_{j_1}\cdots v_{j_t}),
$$
as we wanted to show.

If the matrix $R_{\mathcal I}$ is primitive, then by \autoref{L:positive_matrix} $(\lambda I -R_{\mathcal I})^{-1}$ has positive entries. Hence, if $p_{\mathcal I}> \alpha p_{\mathcal J}$, we have:
$$
  (v_{i_1}\cdots v_{i_t})=  p_{\mathcal I} \: (\lambda I -R_{\mathcal I})^{-1} > \alpha \:  p_{\mathcal J} \: (\lambda I -R_{\mathcal I})^{-1}
$$
$$
  = \alpha \:  p_{\mathcal J} \: (\lambda I -R_{\mathcal J})^{-1} = \alpha\: (v_{j_1}\cdots v_{j_t}),
$$
\end{proof}

Now we can finally use the above results to show that lexicographic representatives in $A_n$ that finish with $a_1$ do not become negligible as $n$ tends to infinity.

\begin{theorem}\label{T:generic_braids_end_with_a1}
Let $P_{n,1}$ be the limit, when $k$ tends to infinity, of the proportion of maximal lexicographic representatives of length $k$ in $A_n$ finishing with $a_1$. Then $P_{n,1}>\frac{1}{8}$ for every $n\geq 1$.
\end{theorem}

\begin{proof}
It is clear that $P_{1,1}=1$, and we already saw that $P_{2,1}=\frac{1}{2}$. So we can assume that $n\geq 3$.

In \autoref{T:automaton} we unveiled the structure of the automaton $\Gamma_n$, and hence of $\Gamma^*_n=\Gamma_n \setminus sh(\Gamma_{n-1})$. Recall from there that we denote $T_{1,n}=\{(1,1,k,\emptyset); \ k=1,\ldots,n\}\subset \mathcal S_n$. From the structure of $\Gamma_n$ we see that the only states in $\Gamma^*_n$ with incoming arrows labeled $a_1$ are those in $T_{1,n}$.  Hence, a lexicographic representative finishing with $a_1$ must end at a state in $T_{1,n}$.

Given a subgraph $\Gamma$ of $\Gamma^*_n$, we will denote $v(\Gamma)$ the sum of the coordinates of $v$  corresponding to the states of $\Gamma$, where $v$ is the normalized Perron-Frobenius eigenvector for $\Gamma^*_n$. Notice that $P_{n,1}=v(T_{1,n})$. Hence, we must show that $ v(T_{1,n})>\frac{1}{8}$.

We will distinguish several subgraphs of $\Gamma^*_n$, and we will describe isomorphisms between some of them. Recall that if $w=(1,j,k,S)$ is a state of $\Gamma^*_n$ not involving the number $n$, we defined $sh_{\bullet}(w)=(1,j',k',S')$, where $i'$, $j'$ and $S'$ are obtained from $i$, $j$ and $S$ by increasing one unit all the numbers involved. We can keep applying $sh_{\bullet}$ as long as the numbers involved are smaller than $n$. To simplify the notation, for every state $w$ we will write $w^{(r)}=sh_{\bullet}^r(w)$ whenever this makes sense, that is, when the numbers involved in $w=(1,j,k,S)$ are at most $n-r$.

We already mentioned the subset of states $T_{1,n}=\{(1,1,k,\emptyset); \ k=1,\ldots,n\}$. We will denote $t_{1,k}=(1,1,k,\emptyset)$, for $k=1,\ldots,n$.  Whenever $k+r\leq n$, we can consider the state $t_{1,k}^{(r)}=(1,1+r,k+r,\emptyset)$.  Let us see which arrows in $\Gamma^*_n$ end at theses states.

The elements in $T_{1,n}$ have many incoming arrows, all labeled $a_1$ (see \autoref{F:Gamma_n}). On the other hand, a state $t_{1,k}^{(1)}$ has all incoming arrows labeled $a_2$, and it belongs to $sh_{\bullet}(\Gamma^*_{n-1})$, which is almost identical to $\Gamma^*_{n-1}$: the inner arrows in $sh_{\bullet}(\Gamma^*_{n-1})$ are precisely the inner arrows in $\Gamma^*_{n-1}$ (shifted one unit), except those labeled $a_1$.  It follows that $sh_{\bullet}(\Gamma^*_{n-1})$ has no inner arrow labeled $a_2$. So there are no inner arrows in $sh_{\bullet}(\Gamma^*_{n-1})$ ending at a state $t_{1,k}^{(1)}$. Hence, the arrows ending at $t_{1,k}^{(1)}$ in $\Gamma^*_{n}$ must come from outside $sh_{\bullet}(\Gamma^*_{n-1})$. By \autoref{T:automaton} we see that the only incoming arrows are either
$$
   t_{1,1} \stackrel{a_2}{\longrightarrow} t_{1,1}^{(1)}
$$
or
$$
   t_{1,k+1} \stackrel{a_2}{\longrightarrow} t_{1,k}^{(1)}
$$
for $k=2,\ldots,n-1$.  See \autoref{F:Gamma_5}, in which we have sketched the case $n=5$.

If $r>1$, as every arrow ending at $t_{1,k}^{(r)}$ is labeled $a_{r+1}$ (and this is neither $a_1$ nor $a_2$), it must be the image under $sh_{\bullet}$ of an arrow in $\Gamma^*_{n-1}$. It follows that the only such arrows are:
$$
  t_{1,1}^{(r-1)} \stackrel{a_{r+1}}{\longrightarrow} t_{1,1}^{(r)}
$$
and
$$
  t_{1,k+1}^{(r-1)} \stackrel{a_{r+1}}{\longrightarrow} t_{1,k}^{(r)}
$$
for $k=2,\ldots,n-r$. See \autoref{F:Gamma_5}.

\begin{figure}
\begin{center}
  \includegraphics{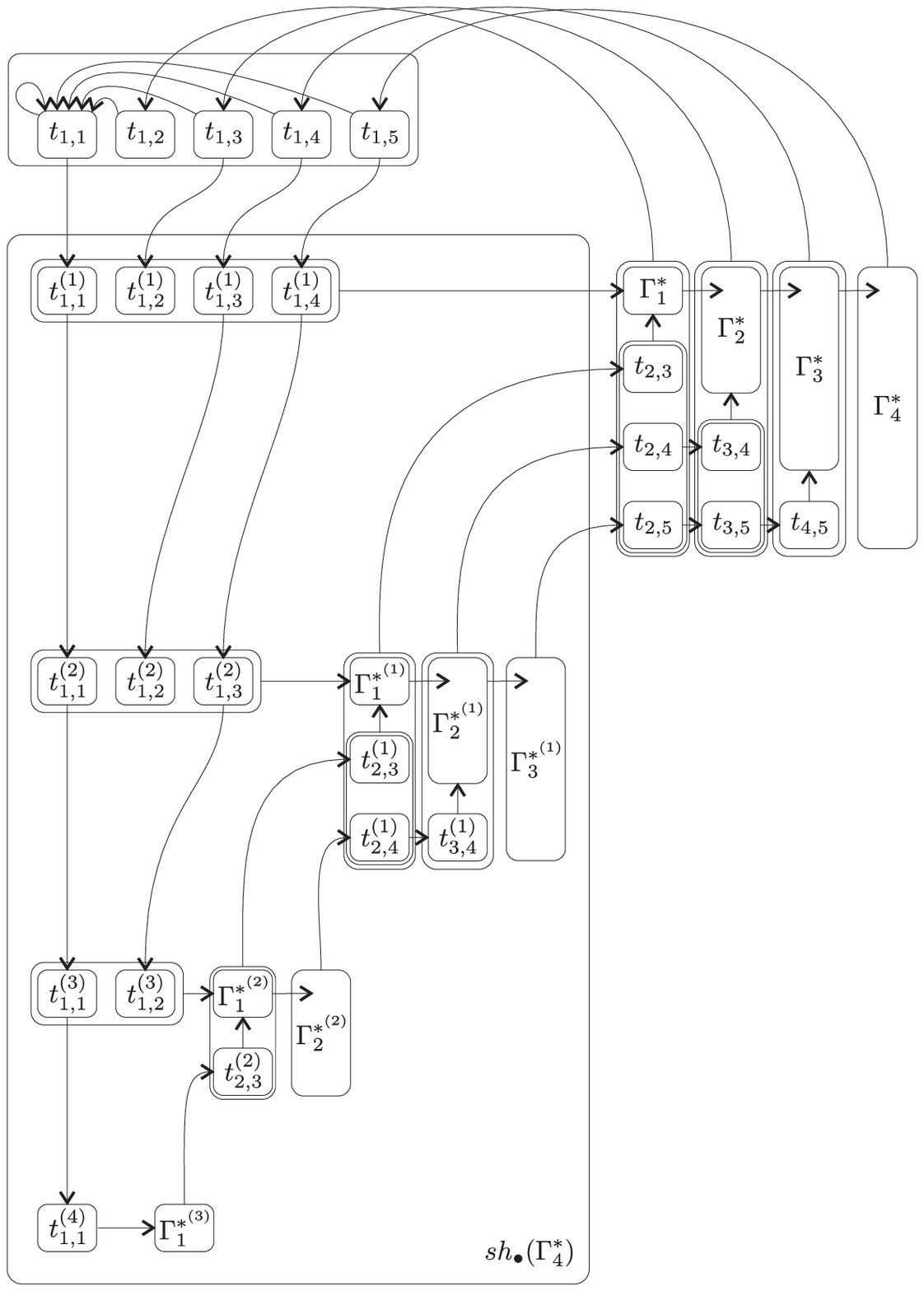}
\end{center}
\caption{Outline of $\Gamma^*_5$, highlighting the transitions which repeat at different parts of the graph.}
\label{F:Gamma_5}
\end{figure}

We then have $\mbox{Pred}(t_{1,1}^{(r)})=\{t_{1,1}^{(r-1)}\}$ and $\mbox{Pred}(t_{1,k}^{(r)})=\{t_{1,k+1}^{(r-1)}\}$, for all $r=1,\ldots,n-1$ and $k=2,\ldots,n-r$. This shows a relation between the coordinates of the Perron-Frobenius eigenvector $v$, associated to these states:
$$
  v(t_{1,1}^{(r-1)})=\lambda_n\; v(t_{1,1}^{(r)}), \qquad  v(t_{1,k+1}^{(r-1)})=\lambda_n\; v(t_{1,k}^{(r)}).
$$
We can collect these states, denoting (by abuse of notation) $T_{1,n}^{(r)}=\{t_{1,k}^{(r)};\ k=1,\ldots,n-r\}$, for $i=0,\ldots,n-1$. All the direct predecessors of the states in $T_{1,n}^{(r)}$ belong to $T_{1,n}^{(r-1)}$, but there is one extra state in $T_{1,n}^{(r-1)}$, namely $t_{1,2}^{(r-1)}$, which is not one of these predecessors, but whose Perron-Frobenius coordinate is nonzero. Therefore:
$$
     v(T_{1,n}^{(r-1)}) > \lambda_n \; v(T_{1,n}^{(r)}),
$$
for $r=1,\ldots,n-1$.

Now consider the states in $\Gamma^*_n$ which belong to neither $T_{1,n}$ nor $sh_{\bullet}(\Gamma^*_{n-1})$. By \autoref{T:automaton}, these states form the subgraphs $T_{2,n},\ldots,T_{n-1,n}$ and $\overline{sh(\Gamma^*_1)}, \ldots,$ $\overline{sh(\Gamma^*_{n-1})}$. In order to simplify the notation, we will just denote $\Gamma^*_m$ the subgraph $\overline{sh(\Gamma^*_m)}$, for $m=1,\ldots,n-1$ (see \autoref{F:Gamma_5}).

Recall that, for $j\geq 2$, we have $T_{j,n}=\{(1,j,k,\{[1,j]\});\ k=j+1,\ldots,n\}$. Let us denote $t_{j,k}=(1,j,k,\{[1,j]\})$, for $2\leq j<k\leq n$. We have already seen in \autoref{T:automaton} that the arrows starting at such a state are precisely the following ones:
$$
   t_{j,k}\stackrel{a_{j+1}}{\longrightarrow} t_{j+1,k}
$$
if $k>j+1$, and
$$
   t_{j,k}\stackrel{a_j}{\longrightarrow} (1,j,j,\{[1,j],[j-1,j]\})\in \Gamma^*_{j-1},\qquad t_{j,k}\stackrel{a_1}{\longrightarrow} (1,1,j,\emptyset)\in T_{1,n}.
$$
These arrows are represented in \autoref{F:Gamma_5}.

We also know that these states have no other incoming arrow, except when $j=2$. The state $t_{2,k}$ (with $k>2$), as we saw in \autoref{T:automaton}, has an incoming arrow labeled $a_2$ from every state in $sh_{\bullet}(\Gamma^*_{n-1})$ whose diagram has a segment $[2,k]$ (and of course a black circle at position 1).  These states are precisely the image under $sh_{\bullet}$ of the states in $\Gamma^*_{k-2}$, and of the states $t_{k-1,l}$, for $l=k,\ldots,n-1$. We will denote $\Gamma^{*^{(1)}}_{m}$ the image under $sh_{\bullet}$ of $\Gamma^*_m$, for $m=1,\ldots,n-2$ (for $m=n-1$ the shifting is not possible), and we will denote $t_{j,k}^{(1)}=sh_{\bullet}(t_{j,k})$, where $k<n$ (for $k=n$ the shifting is not possible). We have then seen that the arrows arriving to $t_{2,k}$ come from every state in $\Gamma^{*^{(1)}}_{k-2}$ and from every state $t_{k-1,l}^{(1)}$, for $l=k,\ldots,n-1$. See \autoref{F:Gamma_5}.

We are now interested in the arrows starting or finishing at some $\Gamma^*_{m}$, which are not inner arrows. This information can also be obtained from \autoref{T:automaton}. There is an arrow labeled $a_2$ from every state $t_{1,k}^{(1)}$ ($k=1, \ldots,n-1$), to the only state $(1,2,2,\{[1,2]\})\in \Gamma^*_1$.  And there is an arrow labeled $a_{m+2}$ from every arrow in $\Gamma^*_{m}$ whose diagram has a square at position $m+1$, to a state in $\Gamma^*_{m+1}$ whose diagram has square at position $m+2$. Also, there is an arrow labeled $a_1$ from every state in $\Gamma^*_{m}$ to $t_{1,m+1}$.

We have then the whole picture of the arrows connecting the states in $\Gamma^*_1,\ldots,\Gamma^*_{n-1}$ and the states $t_{j,k}$ with $2\leq j<k\leq n$.

The important observation now is that every state in $sh_{\bullet}(\Gamma^*_{n-1})$ which is not of the form $t_{1,k}^{(r)}$, is the image under a power of $sh_{\bullet}$ of some of the states described in the previous paragraph. Indeed, the diagram of any element in $sh_{\bullet}(\Gamma^*_{n-1})$ starts with a black circle. Suppose that it has black circles at positions $1,\ldots,r$ but not at $r+1$. If it has a square at position $r+1$, then it is one of the states of the form $t_{1,k}^{(r)}$. Otherwise, it has a segment of the form $[r+1,r+m]$. But then it is equal to $w^{(r)}$, for some state $w$ which either belongs to $\Gamma^*_{m-1}$, or has the form $t_{m,k}$.

Since applying $sh_{\bullet}$ respects the transitions not labeled $a_1$ and increases the indices by 1, it follows that we can describe all transitions between the states in $sh_{\bullet}(\Gamma^*_{n-1})$. They are precisely as described in \autoref{F:Gamma_5}.

In particular, this description allows us to notice that there are several isomorphic subgraphs in $\Gamma^*_n$: Every $\Gamma_m^{*^{(r)}}$ is isomorphic, via $sh_{\bullet}$ to $\Gamma_m^{*^{(r+1)}}$, for every $r=0,\ldots,n-m-2$. We have now all the needed information to compare the values of the coordinates of $v$ corresponding to the different subgraphs of $\Gamma^*_n$.

We claim that, for $r=1,\ldots,n-2$:
$$
   v(\Gamma_m^{*^{(r-1)}}) > \lambda_n \; v(\Gamma_m^{*^{(r)}})
$$
and
$$
  v(t_{j,k}^{(r-1)}) > \lambda_n \; v(t_{j,k}^{(r)}).
$$

In order to show this claim, we will define a partial order on the set of subgraphs
$$
 \mathcal B= \left\{\Gamma_m^{*^{(r)}}\right\}_{r=1,\ldots,n-2 \atop m=1,\ldots,n-r-1} \bigcup \left\{t_{j,k}^{(r)}\right\}_{ r=1,\ldots,n-3 \atop 2\leq j<k\leq n-r},
$$
by saying that $\Gamma$ is smaller than $\Delta$ if there is an arrow from a state of $\Gamma$ to a state of $\Delta$, and taking the transitive closure of this relation. We see in \autoref{F:Gamma_5} that there is a unique minimal subgraph with respect to this partial order, namely $\Gamma_{1}^{*^{(n-2)}}$. We will start by showing the claim for this one.

The subgraph $\Gamma_{1}^{*^{(n-2)}}$ consists of a single state (with an arrow starting at ending at it). The incoming probability of this subgraph is $p_{\mathcal J}=(v(t_{1,1}^{(n-1)}))$. On the other hand, $\Gamma_{1}^{*^{(n-3)}}$ is isomorphic to $\Gamma_{1}^{*^{(n-2)}}$, and its incoming probability is $p_{\mathcal I}=(v(t_{1,1}^{(n-2)})+v(t_{1,2}^{(n-2)})+v(t_{2,3}^{(n-3)}))$.

We already know that $v(t_{1,1}^{(n-2)}) = \lambda_n\; v(t_{1,1}^{(n-1)})$. Hence $p_{\mathcal I}>\lambda_n p_{\mathcal J}$. Therefore, by \autoref{P:subgraphs_comparison}:
$$
   v(\Gamma_1^{*^{(n-3)}}) >  \lambda_n \; v(\Gamma_1^{*^{(n-2)}}).
$$

Now take a subgraph $\Gamma$ in $\mathcal B$, and assume that the claim holds for smaller subgraphs (with respect to the described partial order). We know that there exists a subgraph $\Delta$ of $\Gamma^*_n$ such that $\Gamma=sh_{\bullet}(\Delta)$, where $sh_{\bullet}$ induces a subgraph isomorphism. Moreover, the arrows with target in $\Gamma$ come either from subgraphs of $\mathcal B$ which are smaller than $\Gamma$, or from $T_{1,n}^{(r)}$ for some $r>1$. Hence, for every arrow $w\stackrel{a_i}{\longrightarrow} w'$ with $w\notin \Gamma$ and $w'\in \Gamma$, there exists an arrow $u\stackrel{a_{i-1}}{\longrightarrow} u'$ with $u\notin \Delta$ and $u'\in \Delta$, where $sh_{\bullet}(u)=w$ and $sh_{\bullet}(u')=w'$. Therefore, the `incoming probability' for the graphs $\Gamma$ and $\Delta$ satisfies
$$
   p_{\Delta} > \lambda_n \; p_{\Gamma}.
$$
Now notice that if $\Gamma$ is a single state (with no arrow starting and finishing at it), it is immediate that $v(\Delta)>\lambda_n v(\Gamma)$. Otherwise, $\Gamma$ is isomorphic to $\Gamma^*_m$ for some $m<n-1$. Hence, its incidence matrix is primitive. Therefore, by \autoref{P:subgraphs_comparison}, we have
$$
  v(\Delta)>\lambda_n \;v(\Gamma),
$$
showing the claim.

We can now collect all the information as follows. For $r=0,\ldots,n-2$, denote
$$
   B_r= \left(\bigcup_{i=1}^{n-r-1}{\Gamma_{i}^{*^{(r)}}}\right) \bigcup \left( \bigcup_{k=3}^{n-r}{\bigcup_{j=2}^{k-1}{t_{j,k}^{(r)}}} \right)
$$
These are the subgraphs which are grouped together in \autoref{F:Gamma_5}. By the above arguments, we have shown that
$$
   v(B_0) > \lambda_n^{r}\; v(B_r)
$$
for $r=1,\ldots,n-2$. Hence:
$$
  v(B_1\sqcup\cdots \sqcup B_{n-2}) = v(B_1)+\cdots + v(B_{n-2})
$$
$$
  < \left(\lambda_n^{-1}+\lambda_n^{-2}+\cdots \lambda_n^{-(n-2)}\right) v(B_0) < v(B_0).
$$
Recall that the sequence $\{\lambda_n\}$ is increasing and, as we will see in \autoref{Tbl:computations}, $\lambda_3=2.086\cdots$. Hence, assuming $n>2$, we have $\lambda_n>2$ and the last inequality holds.

On the other hand, we already showed that
$$
  v(T_{1,n})>\lambda_n^{r}\; v(T_{1,n}^{(r)})
$$
for $r=1,\ldots,n-1$. Therefore
$$
  v(T_{1,n}^{(1)}\sqcup\cdots \sqcup T_{1,n}^{(n-1)})= v(T_{1,n}^{(1)})+\cdots + v(T_{1,n}^{(n-1)})
$$
$$
< \left(\lambda_n^{-1}+\lambda_n^{-2}+\cdots \lambda_n^{-(n-1)}\right) v(T_{1,n})< v(T_{1,n}).
$$
But we know that $sh_{\bullet}(\Gamma^*_{n-1})=(B_1\sqcup\cdots \sqcup B_{n-2} \sqcup T_{1,n}^{(1)}\sqcup\cdots \sqcup T_{1,n}^{(n-1)})$. Hence:
$$
   v(sh_{\bullet}(\Gamma^*_{n-1}))<v(B_0)+v(T_{1,n}).
$$
The three subgraphs in the above formula are disjoint, and cover the whole graph $\Gamma^*_n$. Hence:
$$
   v(sh_{\bullet}(\Gamma^*_{n-1}))+v(B_0)+v(T_{1,n}) = v(\Gamma^*_n)=1.
$$
Therefore:
$$
  v(B_0)+v(T_{1,n})>\frac{1}{2}.
$$
Finally, we see that $\mbox{Pred}(T_{1,n})=B_0\sqcup T_{1,n}$. And there is a single arrow (labeled $a_1$), from each state in $B_0\sqcup T_{1,n}$ to a state in $T_{1,n}$. This implies that
$$
  v(B_0)+v(T_{1,n}) = v(B_0\sqcup T_{1,n})=\lambda_n\; v(T_{1,n}).
$$
Hence, as $\lambda_n<4$ (\autoref{T:limit<4}):
$$
  v(T_{1,n})>\frac{1}{2\lambda_n}>\frac{1}{8}.
$$
\end{proof}

\begin{corollary}\label{C:1/32}
Let $P_{n,a_1}$ be the limit, when $k$ tends to infinity, of the proportion of braids $\beta\in A_n$ of length $k$, such that $F_n(\beta)=F_n(a_1)$. Then $P_{n,a_1}>\frac{1}{32}$ for every $n\geq 1$.
\end{corollary}

\begin{proof}
The number $P_{n,a_1}$ is the limit, when $k$ tends to infinity, of the proportion of lexicographic representatives finishing at the state $F_n(a_1)$, that is, at $t_{1,1}$. In other words, $P_{n,a_1}=v(t_{1,1})$.

Now it suffices to notice that $\mbox{Pred}(t_{1,1})=T_{1,n}$. Hence, by \autoref{T:generic_braids_end_with_a1}:
$$
    v(t_{1,1})= \lambda_n^{-1} \; v(T_{1,n}) > \lambda_n^{-1}\;  \frac{1}{8} > \frac{1}{32}.
$$
\end{proof}

\section*{Appendix: an algorithm to generate transition matrices.}

In this paper we have described in detail the automaton $\Gamma_n$ which accepts the language of maximal lexicographic representatives for the braid monoid $A_n$. Thanks to this analysis, we have been able to develop a computer program that generates the primitive transition matrix $R_n$ corresponding to the subgraph $\Gamma_n^*$, from which the automaton $\Gamma_n$ can be easily obtained: just recall that $\Gamma_1=\Gamma_1^*$, and that $\Gamma_{n}$ consists of a copy of $\Gamma_{n-1}$ (that we denoted $sh(\Gamma_{n-1})$) and a copy of $\Gamma_n^*$, together with an arrow (labeled $a_1$) starting at each vertex of $sh(\Gamma_{n-1})$ and pointing to the first vertex of $\Gamma_n^*$. If one is interested in computing the actual automaton, the only information missing will be the labels of the arrows. These labels can be obtained from the proof of \autoref{P:bijection}.

Our main goal here is the computation of the matrix $R_n$, as its Perron-Frobenius eigenvalue corresponds to the growth rate of $A_n$, and its Perron-Frobenius left-eigenvector shows the proportion of maximal lexicographic representatives finishing at each state.

Recall (see \autoref{F:Gamma_5}) that the automaton $\Gamma_n^*$ is very similar to the automaton $sh_{\bullet}(\Gamma_n^*)$, the only difference being that the arrows labeled $a_1$ in $\Gamma_n^*$ point to the states $t_{1,i}$, and the arrows labeled $a_2$ in $sh_{\bullet}(\Gamma_n^*)$ point to either $\Gamma_1^*$ or to the states $t_{2,i}$. We will then use the same algorithm in both cases, with a flag (named `$closed$') indicating whether we are computing $\Gamma_n^*$ ($closed=1$) or $sh_{\bullet}(\Gamma_n^*)$ ($closed=0$).

We will order the rows of $R_n$ (that is, the states of $\Gamma_n^*$) in a systematic way. See \autoref{F:Gamma_5}. The first $n$ rows will correspond to the states $t_{1,1},\ldots,t_{1,n}$. The next rows will correspond to $sh_{\bullet}(\Gamma_{n-1}^*)$. And the final rows are ordered as they appear in \autoref{F:Gamma_5}, by columns: first $\Gamma_1^*, t_{2,3}, t_{2,4},\ldots, t_{2,n}$, then the states in $\Gamma_2^*, t_{3,4},\ldots, t_{3,n}$, and so on, until the final rows which correspond to the states of $\Gamma_{n-1}^*$. Inside $sh_{\bullet}(\Gamma_{n-1}^*)$ (whose states correspond to those of $\Gamma_{n-1}^*$), and inside each copy of $\Gamma_i^*$ for $i=1,\ldots,n-1$, the states will be ordered in the same way, by recurrence.

Notice that we can easily compute the position of any state in \autoref{F:Gamma_5} thanks to \autoref{T:numer_of_states}, as the number of states of $\Gamma_n^*$ is precisely $s_n-s_{n-1}$. We then need to pre-compute the numbers $s_1,\ldots,s_n$, where $s_i$ is the number of states in $\Gamma_i$. We will not write this explicitly, as it follows from the single formula in \autoref{T:numer_of_states}. Or, even better, from the formula in \autoref{P:recurrence_relation}. For $i=1,\ldots,n$, let $s_i^*=s_i-s_{i-1}$ be the number of states in $\Gamma_i^*$. We can assume we have computed all these numbers, for $i=1,\ldots,n$.

We will define a procedure ${\tt Submatrix}(R,j,H,s,closed)$, that will take a sparse matrix $R$ (that will eventually become $R_n$) and will compute the submatrix corresponding to either $\Gamma_{j}^*$ (if $closed=1$) or $sh_{\bullet}(\Gamma_j^*)$ (if $closed=0$). It will place that submatrix starting at position $(1+s,1+s)$ of $R$ ($s$ is just the shifting indicating where to place the submatrix).  The input $H$ is a vector that should be pre-computed, and contains the exact positions of the source states of the horizontal arrows in \autoref{F:Gamma_5} going from $\Gamma_{i}^*$ to $\Gamma_{i+1}^*$, for $i=1,\ldots,j-2$.

The numbers in $H$ are very particular, and the same numbers can be used for $i=1,\ldots,j-2$. That is why they can be pre-computed. By the proof of \autoref{P:bijection} we know that the mentioned arrows go from the segment configuration $(1,i+1,i+1,S)$ to the segment configuration $(1,i+2,i+2,S')$, where $S=\{[1,i+1],[i_2,i+1],\ldots,[i_t,i+1]\}$ and $S'=\{[1,i+2],[i_2,i+2],\ldots,[i_t,i+2]\}$. There are exactly $2^{i-1}$ such arrows (the set $\{i_2,\ldots,i_t\}$ can be any subset of $\{2,\ldots,i\}$). It turns out that the source states of these arrows correspond to the following (we will use the whole diagram in \autoref{F:Gamma_5} as a picture of $\Gamma_i^*$): first, the state appearing at the bottom-left of \autoref{F:Gamma_5} ($t_{1,1}^{(i-1)}$); then the only state of ${\Gamma_1^*}^{(i-2)}$; then the same two described states, but inside ${\Gamma_2^*}^{(i-3)}$; then the same four described states, but inside ${\Gamma_3^*}^{(i-4)}$; and so on. These positions configure the numbers in $H$. On the other hand, the target states of the arrows from $\Gamma_{i-1}^*$ to $\Gamma_i^*$ are actually a subset of the source states from $\Gamma_i^*$ to $\Gamma_{i+1}^*$; namely, they correspond to the entries of $H$ placed in odd positions.

Therefore, since the exact positions can be computed from the structure revealed in \autoref{F:Gamma_5} and the knowledge of the numbers $s_i^*$, the vector $H$ can be determined as follows:

\begin{algorithm}[ht]
\caption{Computing the vector $H$ used in the computation of $R_j$}
\label{A:H}
\small\begin{algorithmic}[1]
\REQUIRE{Integer $j\geq 2$.}
\ENSURE{A vector $H$ of length $2^{j-1}$ containing the positions in $\Gamma_{j}^*$ of the source states of the horizontal arrows in \autoref{F:Gamma_5}}\smallskip
\STATE{$H:=[1]$}
\FOR{$i := 2$ to $j$}
  \STATE{$x:=s_i^*-s_{i-1}^*-i$}
  \STATE{If $H=[h_1,\ldots,h_r]$, set $H':=[h_1+x, \ldots, h_r+x]$}
  \STATE{$H:=$concatenation of $H$ and $H'$.}
\ENDFOR
\RETURN{$H$}
\end{algorithmic}\vskip0.5ex
\end{algorithm}

As an example, if $j=4$, the algorithm will produce $H=[0, 1, 6, 7, 21, 22, 27, 28]$. In $\Gamma_5^*$ (see \autoref{F:Gamma_5}), if we consider that the position $0$ corresponds to the bottom state $t_{1,1}^{(4)}$, then these positions correspond respectively to the state $t_{1,1}^{(4)}$, the only state in ${\Gamma_1^*}^{(3)}$, the two bottom states in ${\Gamma_2^*}^{(2)}$, and the four states repeating the same pattern in ${\Gamma_3^*}^{(1)}$.  Notice that for $j=3$ we have $H=[0,1,6,7]$. So having computed $H$ for some $j$, we have computed it for all smaller indices.

Inside each $\Gamma_i^*$, one needs to locate the bottom state (corresponding to 0). But this is easy, as it is the state number $i+1 \choose 2$ in $\Gamma_i^*$.  This means that in any automaton in which we have horizontal arrows from $\Gamma_3^*$ to $\Gamma_4^*$, the arrows will start at positions $\left[0+{4\choose 2},1+{4\choose 2},6+{4\choose 2},7+{4\choose 2}\right]=[6,7,12,13]$ (inside $\Gamma_3^*$), and will finish at positions $\left[0+{5\choose 2},6+{5\choose 2},21+{5\choose 2},27+{5\choose 2}\right]=[10,16,31,37]$ (inside $\Gamma_4^*$).

With this information, and following the pattern of \autoref{F:Gamma_5}, we provide in Algorithm~\ref{A:R_n} the routine that allows to compute $R_n$.

\begin{algorithm}[!ht]
\caption{{\tt Submatrix}$(R,j,H,s,closed)$.}
\label{A:R_n}
\small\begin{algorithmic}[1]
\REQUIRE{Sparse matrix $R$, integer $j\geq 1$, vector of integers $H$, integer $s\geq 0$, bit $closed\in\{0,1\}$}
\ENSURE{Transforms the matrix $R$, inserting 1's in the submatrix of size $s_j^*\times s_j^*$ whose upper left corner is $(1+s,1+s)$, at the positions corresponding either to the arrows of $\Gamma_j^*$ (if $closed=1$), or to the arrows starting at $s_{\bullet}(\Gamma_j^*)$ (if $closed=0$).}
\smallskip
\IF{$closed=1$}
  \FOR{$i:=1$ to $j$}
    \STATE{$R[i+s, 1+s] := 1$ \hfill (Arrows from $t_{1,i}$ to $t_{1,1}$)}
  \ENDFOR
\ELSE
  \FOR{$i:=1$ to $j$}
    \STATE{$R[i+s, 1+s_j^*+s] := 1$ \hfill (Arrows from $t_{1,i}^{(1)}$ to $\Gamma_1^*$)}
  \ENDFOR
\ENDIF
\IF{$j>1$}
  \STATE{$R[1+s, j+1+s] := 1$ \hfill (Arrow from $t_{1,1}$ to $t_{1,1}^{(1)}$)}
\ENDIF
\FOR{$i:=3$ to $j$}
  \STATE{$R[i+s, j+i-1+s] := 1$ \hfill (Arrows from $t_{1,i}$ to $t_{1,i}^{(1)}$)}
\ENDFOR
\STATE{Call {\tt Submatrix}$(R, j-1, H, s+j, 0)$ \hfill (Arrows starting at $sh_{\bullet}(\Gamma_{j-1}^*)$)}
\STATE{$s' := s+j+s_{j-1}^*$ \hfill ($s'+1=$position of $\Gamma_1^*$)}
\FOR{$i:=1$ to $j-1$}
  \STATE{Call {\tt Submatrix}$(R, i, H, s', 1)$ \hfill (Arrows in $\Gamma_{i}^*$)}
  \IF{$i=1$}
    \FOR{$k:=1$ to $j-2$}
      \STATE{$R[s'+1+k,s'+1]:=1$ \hfill (Arrows from $t_{2,k+2}$ to $\Gamma_1^*$)}
    \ENDFOR
  \ELSE
    \FOR{$k:=1$ to $j-i-1$}
      \STATE{$R[s'+s_i^*+k,s'+{i+1\choose 2}+1]:=1$ \hfill (Arrows from $t_{i+1,i+k+1}$ to $\Gamma_{i}^*$)}
    \ENDFOR
  \ENDIF
  \FOR{$k:=2$ to $j-i-1$}
    \STATE{$R[s'+s_i^*+k,s'+s_i^*+k+s_{i+1}^*+j-i-2]:=1$  \hfill (Arrows from $t_{i+1,i+k+1}$ to $t_{i+2,i+k+1}$)}
  \ENDFOR
  \IF{$closed=1$}
    \FOR{$k:=s'+1$ to $s'+s_i^*+j-i-1$}
      \STATE{$R[k,s+i+1]:=1$ \hfill (Arrows from all states of $\Gamma_i^*$ and $t_{i+1,r}$ to $t_{1,i+1}$)}
    \ENDFOR
  \ELSE
    \FOR{$k:=s'+1$ to $s'+s_i^*+j-i-1$}
      \STATE{$R[k,s+s_j^*+i+1]:=1$   \hfill (Arrows from all states of ${\Gamma_i^*}^{(1)}$ and $t_{i+1,r}^{(1)}$ to $t_{2,i+2}$)}
    \ENDFOR
  \ENDIF
  \IF{$i<j-1$}
    \FOR{$k:=1$ to $2^{i-1}$}
      \STATE{$R[s'+{i+1\choose 2}+H[k],s'+s_i^*+j-i-1+{i+2\choose 2}+H[2k-1]]:=1$  \hfill (Arrows from $\Gamma_i^*$ to $\Gamma_{i+1}^*$)}
    \ENDFOR
  \ENDIF
  \STATE{$s':=s'+s_i^*+j-i-1$}
\ENDFOR
\end{algorithmic}\vskip0.5ex
\end{algorithm}

Therefore, in order to compute the transition matrix $R_n$, one just needs to do the following:
\begin{enumerate}

\item Compute $s_1^*,\ldots,s_n^*$, using \autoref{P:recurrence_relation} (where $s_i^*=s_i-s_{i-1}$).

\item Compute $H$ using Algorithm~\ref{A:H}, taking $j=n-1$.

\item Initialize $R$ as a $s_n^*\times s_n^*$ sparse zero matrix.

\item Call {\tt Submatrix}$(R,n,H,0,1)$.

\end{enumerate}

We know that the number of rows of $R_n$ grows exponentially with respect to $n$, so we cannot compute these matrices for big values of $n$; however, we can provide the computations for $n=2,\ldots,9$. We compute in this case the matrix $R_n$, the Perron-Frobenius eigenvalue $\lambda_n$ (which is the growth rate of $A_n$), the first entry of the Perron-Frobenius eigenvector whose coordinates add up to 1 (which is equal to $P_{n,a_1}$), and the sum of the first $n$ entries of the same eigenvector (which is equal to $P_{n,1}$). The results of this computation are shown in \autoref{Tbl:computations}.

\begin{table}[!ht]\begin{center}
\begin{tabular}{@{}r@{\quad }r|ccc@{\,}}
         &       & $\lambda_n$ &  $P_{n,a_1}$ & $P_{n,1}$ \\
  \hline\rule{0pt}{11pt}
         &    2  & 1.61803398874989535 & 0.309016994387306732 & 0.5 \\
         &    3  & 2.08679122278138296 & 0.179072361848063216 & 0.3736866329 \\
         &    4  & 2.39485036123379746 & 0.134155252415486176 & 0.3212817547 \\
     n   &    5  & 2.59937733237127854 & 0.113418385255364101 & 0.2948171798 \\
         &    6  & 2.73962959897194480 & 0.102094618000846169 & 0.2797014374 \\
         &    7  & 2.83910705543066832 & 0.095188754079773799 & 0.2702510632 \\
         &    8  & 2.91185367833772002 & 0.090638078480376610 & 0.2639248222 \\
         &    9  & 2.96648976449784296 & 0.087464812090583224 & 0.2594634699
\end{tabular}\end{center}\vspace*{-3ex}
\caption{Growth rates for $A_n$, the proportion of lexicographic representatives in $A_n$ finishing at the initial state of $\Gamma_n^*$, and the proportion of lexicographic representatives finishing with $a_1$.}\label{Tbl:computations}
\end{table}

It can be checked that the values of $\lambda_n$ coincide with the inverses of the smallest real roots of the polynomials $H_n(x)$ which determine the growth series $\frac{1}{H_n(x)}$ of the monoid $A_n$ (see~\cite{FG}). In other words, the computed values $\lambda_n$ correspond, as they should, to the radius of convergence of the growth series of $A_n$.

On the other hand, one can see that the values of $P_{n,a_1}$ are decreasing, and greater than $1/32=0.03125$ (as stated in \autoref{C:1/32}), and that the values of $P_{n,1}$ are also decreasing, and greater than $1/8=0.125$ (as stated in \autoref{T:generic_braids_end_with_a1}). Notice that, as expected, the values in the third column of \autoref{Tbl:computations} are precisely the product of the values in the first two columns. That is, $P_{n,1}=\lambda_n P_{n,a_1}$.

\begin{remark}
In~\cite{FG} we show, using the results in this paper, that the sequence $\{\lambda_n\}_{n\geq 1}$ tends to $3.233636\ldots$ Therefore, we have $\lambda_n<3.233637$. This allows to improve our lower bounds, replacing $4$ by $3.233637$, and obtaining:
$$
      P_{n,a_1}>0.0478 \qquad P_{n,1}>0.1546
$$
\end{remark}


{\bf Ram\'on Flores.}\\
ramonjflores@us.es\\
Depto. de Geometr\'{\i}a y Topolog\'{\i}a. Instituto de Matem\'aticas (IMUS). \\
Universidad de Sevilla.  Av. Reina Mercedes s/n, 41012 Sevilla (Spain).

\medskip

{\bf Juan Gonz\'alez-Meneses.}\\
meneses@us.es\\
Depto. de \'Algebra. Instituto de Matem\'aticas (IMUS). \\
Universidad de Sevilla.  Av. Reina Mercedes s/n, 41012 Sevilla (Spain).

\end{document}